\documentclass[a4paper, final]{amsart}
\pdfoutput=1
\RequirePackage{standalone}

\RequirePackage[utf8]{inputenx}
\RequirePackage[T1]{fontenc}

\RequirePackage{lmodern}           
\RequirePackage[scaled]{beramono}  
\RequirePackage[final]{microtype}  
\RequirePackage{siunitx}           
\RequirePackage{listings}          
\RequirePackage{url}               
\RequirePackage{verbatim}
\RequirePackage{csquotes}

\RequirePackage{amssymb}    
\RequirePackage{amsthm}     
\RequirePackage{thmtools}   
\RequirePackage{mathtools}  
\RequirePackage{mathrsfs}   
\RequirePackage{dsfont}     
\RequirePackage{cancel}     
\RequirePackage{stmaryrd}   
\RequirePackage{calligra}	

\RequirePackage[dvipsnames,
svgnames,
cmyk]{xcolor}     
\RequirePackage{graphicx}         
\graphicspath{{figures/}}
\RequirePackage{tikz}             
\usetikzlibrary{calc}
\usetikzlibrary{intersections}
\usetikzlibrary{decorations.markings}
\RequirePackage{tikz-cd}		  

\RequirePackage[notref, notcite]{showkeys}               

\RequirePackage{xspace}         
\RequirePackage{textcomp}       

\RequirePackage[backend = biber, style = numeric, ibidtracker=true, maxbibnames=4]{biblatex}
\RequirePackage{varioref}
\RequirePackage{hyperref}
\RequirePackage[nameinlink, capitalize, noabbrev]{cleveref}

\declaretheorem[style = plain, numberwithin = section]{theorem}
\declaretheorem[style = plain,      sibling = theorem]{corollary}
\declaretheorem[style = plain,      sibling = theorem]{lemma}
\declaretheorem[style = plain,      sibling = theorem]{proposition}

\declaretheorem[style = definition, sibling = theorem]{definition}
\declaretheorem[style = definition, sibling = theorem]{example}

\declaretheorem[style = remark,     sibling = theorem]{remark}
\declaretheorem[style = remark,     sibling = theorem]{remarks}
\crefname{observation}{Observation}{Observations}
\Crefname{observation}{Observation}{Observations}
\crefname{conjecture}{Conjecture}{Conjectures}
\Crefname{conjecture}{Conjecture}{Conjectures}
\crefname{assumption}{Assumption}{Assumption}
\Crefname{assumption}{Assumption}{Assumption}
\crefname{notation}{Notation}{Notations}
\Crefname{notation}{Notation}{Notations}
\crefname{diagram}{Diagram}{Diagrams}
\Crefname{diagram}{Diagram}{Diagrams}

\DeclareMathOperator{\Spec}{Spec}				
\DeclareMathOperator{\Proj}{Proj}				

\DeclareMathOperator{\coker}{coker}				
\DeclareMathOperator{\Hom}{Hom}					
\DeclareMathOperator{\sheafHom}					
{
	\mathscr{H}\text{\kern -5.2pt {\calligra\large om}}\,
}
\DeclareMathOperator{\Sym}{Sym}					
\newcommand{\gitquot}{\mathord{					
		\mathchoice{/\mkern-6mu/}
		{/\mkern-6mu/}
		{/\mkern-5mu/}
		{/\mkern-5mu/}}}
\newcommand{\N}{\mathbb{N}}    						
\newcommand{\Z}{\mathbb{Z}}    						
\newcommand{\Q}{\mathbb{Q}}    						
\newcommand{\C}{\mathbb{C}}    						
\renewcommand{\P}{\mathbb{P}}  						

\newcommand{\GL}{\mathrm{GL}}						
\newcommand{\SL}{\mathbf{SL}}						

\newcommand{\cat}[1]{{\normalfont\mathsf{#1}}}	

\DeclareMathOperator{\Supp}{Supp}				

\newcommand{\Coh}[1]{\cat{Coh}({#1})}			
\newcommand{\ECoh}[2]{\cat{Coh}_{#1}({#2})}		

\newcommand{\ol}[1]{\overline{#1}}							
\newcommand{\shf}[1]{\mathscr{#1}}							
\newcommand{\OO}{\mathcal{O}}								
\DeclareMathOperator{\id}{id}								

\DeclareMathOperator{\im}{im}					
\newcommand{\Hilb}{\operatorname{Hilb}}						
\newcommand{\isoto}{\xrightarrow{\sim}}						
\newcommand{\injto}{\xhookrightarrow{}}						

\renewcommand{\setminus}{\smallsetminus}

\newcommand{\dual}{{}^\vee}									
\renewcommand{\tilde}{\widetilde}

\newcommand{\con}{_{\mathrm{con}}}							

\newcommand{\ie}{\leavevmode\unskip, i.e.,\xspace}
\newcommand{\eg}{\leavevmode\unskip, e.g.,\xspace}

\usepackage{enumerate}
\usepackage{enumitem}

\DeclareMathOperator{\tail}{\mathrm{t}}
\DeclareMathOperator{\head}{\mathrm{h}}
\usetikzlibrary{chains}

\renewcommand{\hat}{\widehat}

\newcommand{\sh}{\widehat}

\mathchardef\mhyphen="2D
\newcommand{\ngammahilb}{n\Gamma\mhyphen\Hilb(\C^2)}		

\usetikzlibrary{arrows,decorations.pathmorphing,decorations.pathreplacing,positioning,shapes.geometric,shapes.misc,decorations.markings,decorations.fractals,calc,patterns}

\usetikzlibrary{chains}

\tikzset{>=stealth',
	cvertex/.style={circle,draw=black,inner sep=1pt,outer sep=3pt},
	vertex/.style={circle,fill=black,inner sep=1pt,outer sep=3pt},
	star/.style={circle,fill=yellow,inner sep=0.75pt,outer sep=0.75pt},
	tvertex/.style={inner sep=1pt,font=\scriptsize},
	gap/.style={inner sep=0.5pt,fill=white}}

\bibliography{bibliography}
\title[Quiver varieties and framed sheaves on singularities]{Quiver Varieties and Framed Sheaves on Compactified Kleinian Singularities}

\author{Søren Gammelgaard}           
\address{Scuola Internazionale Superiore di Studi Avanzati,
	Via Bonomea 265,
	34136 Trieste,
Italy}
\email{sgammelg@sissa.it}

\begin{document}

	\setcounter{tocdepth}{1}

	\begin{abstract}
Consider a Kleinian singularity $ \C^2/\Gamma $, where $ \Gamma $ is a finite subgroup of $ \SL_2(\C) $. In this paper, we introduce a natural stack compactifying the singularity by adding a smooth stacky divisor, and we show that sets of framed sheaves on this stack satisfying certain additional criteria are closely related to a class of Nakajima quiver varieties.
This partially extends our previous work on punctual Hilbert schemes of Kleinian singularities.
\end{abstract}
	\maketitle                  

		\tableofcontents            
	\section{Introduction}\label{chap:Intro}

  Nakajima quiver varieties, introduced in \cite{Nakajima94}, are a versatile class of objects: they are normal and irreducible, when smooth, they are always hyperkähler, and when singular, they always have symplectic singularities \cite{BelSche}. There has been much recent interest in them, especially as they provide constructions of several moduli spaces (see \eg \cite{Kuznetsov07}, \cite{Kuznetsovetal}, ). In previous work \cite{CGGS, CGGS2}, we used Nakajima quiver varieties to investigate moduli spaces attached to Kleinian singularities\ie singular surfaces of the form $ \C^2/\Gamma $, for $ \Gamma $ a finite subgroup
  of $ \SL_2(\C) $. %
In \cite{CGGS}, we showed that there is an isomorphism\footnote{
	The proof of the isomorphism of (the reduced subschemes underlying) punctual Hilbert schemes with certain Nakajima quiver varieties relies on \cite[Lemma 4.3]{CGGS}, the proof of which is not correct. We thank Yehao Zhou for pointing this out to us. The published arguments are however strong enough to establish a morphism $  \mathfrak{M}_{\theta_{0}}(1, n\delta)\to \Hilb^n{(\C^2/\Gamma)}_{red} $ which gives a bijection of closed points. With some additional arguments, that this morphism is an isomorphism can nevertheless be shown to hold~\cite{CrawForthcoming}.
} from (the reduced subscheme of) a punctual Hilbert scheme $ \Hilb^n(\C^2/\Gamma)_{red}  $ of a Kleinian singularity to a specific Nakajima quiver variety $ \mathfrak{M}_{\theta_{0}}(1, n\delta) $.  We later showed that a larger class of Nakajima quiver varieties can similarly be interpreted as `equivariant Quot schemes' (see \cite[Section 3]{CGGS2}).  In this paper, we continue this line of investigation, by defining various sets of (isomorphism classes of) framed sheaves on a stack compactification of $ \C^2/\Gamma $, and showing that each of these sets is in canonical bijection with the closed point sets of appropriate Nakajima quiver varieties.

Once and for all, fix a finite subgroup $\Gamma\subset \SL(2,\mathbb{C})$.
Starting with $ \Gamma $, we shall use the McKay correspondence \cite{McKay80} to construct a quiver $ Q_\Gamma $, with a vertex set $ Q_{\Gamma, 0} $ corresponding to the irreducible representations of $ \Gamma $.  
With this fixed quiver $ Q_\Gamma $, a Nakajima quiver variety $ \mathfrak{M}_{\theta}(\mathbf{w}, \mathbf{v}) $ is a GIT quotient -- we detail the construction in \cref{sec:McKayQuiver} -- depending on a \emph{stability parameter} $ \theta$, a \emph{framing dimension vector} $ \mathbf{w}\in \Z_{\ge 0}^{Q_{\Gamma, 0}} $, and a \emph{dimension vector} $ \mathbf{v}\in \Z_{\ge 0}^{Q_{\Gamma, 0}} $. In our cases, however, we shall always choose the framing vector to having a single nonzero element. We will thus speak of it as the \emph{framing rank} $ r\coloneqq \mathbf{w}_0$, and denote our Nakajima quiver varieties as $ \mathfrak{M}_\theta(r, \mathbf v) $.  

The space of stability parameters is equipped with a wall-and-chamber structure, and it is known (see \eg \cite{VaVa, Wang99}) that for a stability parameter $ \theta $ lying in a specific chamber $ C^+ $ and a special dimension vector $ n\delta $, the quiver variety $ \mathfrak{M}_\theta(1, n\delta) $ is isomorphic to the `equivariant Hilbert scheme' $ \ngammahilb $, a specific subscheme of the $ \Gamma $-fixed locus $ \Hilb(\C^2)^\Gamma $ of the punctual Hilbert scheme of the plane itself. The quiver variety $ \mathfrak{M}_{\theta_{0}}(1, n\delta) $, mentioned above, is obtained by varying the stability parameter $ \theta $ from the chamber $ C^+ $ into a specific ray in the closure $ \ol{C^+} $. 

	With the same parameter $ \theta_0 $, we aim to describe the quiver variety $ \mathfrak{M}_{\theta_0}(r, n\delta)$ for any pair of positive integers $ r,n $. As indicated, this quiver variety will turn out to parametrise a set $ Y_{r,n} $ of isomorphism classes of framed sheaves on a projective {stack} $ \mathcal X $, of which the singularity $ \C^2/\Gamma $ is an open substack. We will construct $ \mathcal X $ as a quotient of $ \P^2 $ by $ \Gamma $, glued together from two patches, where we on one patch take the `stack-theoretic' quotient, and on another the `scheme-theoretic' quotient, see \cref{sec:weirdQuotient} for the full construction.
	There is a morphism of stacks $ \P^2\to \mathcal X $, and in particular the image of the divisor $ l_\infty \coloneqq \P^2\setminus \C^2 $ is a (stacky) divisor $ d_\infty = [l_\infty/\Gamma] $. 

We define the sets of framed sheaves we are interested in as follows: (here $ \shf I_{d_\infty} $ is the ideal sheaf of $ d_\infty $ in $ \mathcal X $):
\begin{definition}\label{def:setYrn}
	Let $ Y_{r,n} $ be the set of isomorphism classes of pairs $ (\shf F, \phi_{\shf F}) $, where $ \shf F $ is a torsion-free coherent $ \OO_{\mathcal X} $-module of rank $ r $, with $ \dim H^1(\mathcal X, \shf F\otimes \shf I_{d_\infty})= n $, and where $ \phi_{\shf F} $ is an isomorphism  $\phi_{\shf F }\colon \shf F|_{d_\infty}\isoto \OO_{d_\infty}^{\oplus r} $.

\end{definition}
	The pair $ (\shf F, \phi_{\shf F}) $ is a framed sheaf. The set $ Y_{1,n} $ also carries a canonical bijection to the set of closed points $ \Hilb^n(\C^2/\Gamma)(\C) $.
	
Our main result is then:
\begin{theorem}\label{thm:MainWeaker}
	There is a canonical bijection between $ Y_{r,n} $ and the set of closed points of the Nakajima quiver variety $ \mathfrak{M}_{\theta_0}(r, n\delta) $.
\end{theorem}

\begin{remark}\label{rmk:MainExtend}
	It is possible to improve \cref{thm:MainWeaker} by exhibiting a moduli space of which $ Y_{r,n} $ is the set of closed points, and showing that the mentioned bijection comes from a genuine scheme morphism. We will do this in a later paper. However, we remain (as of writing) unable to extend the morphism to an \emph{isomorphism}.
\end{remark}

	Quiver varieties built from a framed McKay quiver with framing rank $ r>1 $ have been investigated before. The wall-and-chamber structure on the parameter space for $ \theta $ was explicitly computed by Bellamy, Craw, and Schedler (\cite{BelCra}, \cite[Theorem 4.18]{BellamyCrawSchedler}) for every framing rank $ r $ and dimension vector $\mathbf v $. 
	For $ \Gamma $ a trivial group, it is known \cite{Nakajimabook} that the quiver variety $ \mathfrak{M}_{\theta}(r, n\delta) $ is isomorphic to the moduli space of framed torsion-free sheaves on $ \P^2 $ of rank $ r $ and second Chern class $ n $, framed along the divisor $ l_\infty$.  With a nontrivial $ \Gamma $, 
	there is a highly relevant result of Varagnolo and Vasserot (see \cref{thm:VaVa}) stating that, for a generic stability parameter $ \theta$ lying in the chamber $ C^+$, there is a canonical bijection of $\mathfrak{M}_{\theta}(r, n\delta)(\C)$ with the the set of isomorphism classes of framed coherent $ \Gamma $-equivariant sheaves of rank $ r $ on $ \P^2 $, satisfying certain further conditions, see \cref{def:Xrv} for the complete definition. For $ r=1 $, we can identify this space with $ \ngammahilb $.

Another description of Nakajima quiver varieties with framing rank $ r>1 $ was given by Nakajima \cite{NakALE}. He constructs a (complex analytic) orbifold compactification $ \hat{{\mathcal X}} $ of the minimal resolution $ \hat{\C^2/\Gamma} $ of $ \C^2/\Gamma $, and shows that there are Nakajima quiver varieties $ \mathfrak{M}_{\theta^{-}}(r, \mathbf{v}) $ isomorphic to moduli spaces of framed sheaves on $ \hat{\mathcal{X}} $. Here $ \theta^- $ is a stability parameter in a specific chamber $ C^{-} $ in the space of stability parameters, and the stack $ \hat{\mathcal{X}} $ is constructed from $ \hat{\C^2/\Gamma} $ in a similar manner to how we shall construct $ \mathcal{X} $ from $ \C^2/\Gamma $. When $ r=1 $, $ \mathfrak{M}_{\theta^{-}}(1,\mathbf{v}) $ is isomorphic to the Hilbert scheme of points on $ \hat{\C^2/\Gamma} $, a result due to Kuznetsov~\cite{Kuznetsov07}.
A similar construction of a stack compactification of $ \hat{\C^2/\Gamma} $ also appears in a conjecture of Bruzzo, Sala, Pedrini, and Szabo, see \cite[Conjecture 4.14]{BPSSz}.

One possible further direction of investigation, not pursued here, is to determine formulae for the generating series $ Z_r(q)= \sum_{m=0}^\infty \chi(\mathfrak{M}_{\theta_{0}}(r, m\delta))q^m $. This has been done for \emph{smooth} Nakajima quiver varieties \cite{Hausel}, even extended to motivic versions \cite{Wyss}. Finding such generating series for singular Nakajima quiver varieties -- such as those we focus on in this paper -- turns out to be more difficult. To the author's knowledge, the only known result is a formula for $ Z_1(q) $, proved by Nakajima in \cite{Nak20}, using results of \cite{CGGS} and \cite{GNS}.

\subsection*{Acknowledgements.}
This paper is adapted from my DPhil thesis (\cite{ID}), defended in September 2022. I am grateful to my supervisor, Balázs Szendr\H{o}i, for his guidance and encouragement while researching this, and his thorough reading of a preliminary version of this paper. Together with Ádám Gyenge and Alastair Craw, I am grateful for our fruitful collaborations that stimulated this work, and further conversations on quiver varieties. I also thank Yehao Zhou, who pointed out a gap in a proof in our previous work, the discussion of which cleared up this work as well. Finally, I thank Andrea Ricolfi for his comments on a previous version of this paper.

I was supported by an Aker Scholarship and by a SISSA Mathematical Fellowship.

\subsection*{Conventions.}
We work throughout over $\C$. In particular, all schemes are  $\C$-schemes and all tensor products are taken over $\C$ unless otherwise indicated. We often use the following partial order on dimension vectors: for any $n\in \mathbb{N}$ and for $u=(u_i),v=(v_i)\in \mathbb{Z}^n$, we define $u\le v$ if $u_i\le v_i$ for all $1\leq i\leq n$. 

We will write $ \Coh{X} $ for the category of coherent sheaves on a scheme (or stack) $ X $.
If $ X $ is equipped with an action of some group $ G $, we write $ \ECoh{G}{X} $ for the category of $ G $-equivariant coherent sheaves on $ X $. 
$ [X/G] $ always denotes the stack-theoretic quotient of $ X $ by $ G $, and $ X/G $ the scheme-theoretic quotient, if it exists. 
Given a sheaf $ \shf F $ on a scheme (or stack) $ X $, we write $ h^i(X, \shf F) $ for $ \dim H^i(X, \shf F) $.

	\section{Preliminaries}\label{chap:BackgroundQuivers}
In this section we collect the necessary background material on quiver varieties and projective stacks.
We start by describing how a finite group $ \Gamma\subset \SL_2(\C) $ can be used to define a \emph{McKay quiver}. We then define the Nakajima quiver varieties, in fact we will define them in three equivalent ways. Our Nakajima quiver varieties will depend on a stability parameter, varying in a wall-and-chamber structure. We define a chamber and ray in it that will be of interest later, and introduce a notion of `concentrated module'. We also give a definition of a \emph{framed sheaf} on a projective Deligne-Mumford stack.

\subsection{Quiver varieties from framed McKay quivers}\label{sec:McKayQuiver}
Recall our finite group $\Gamma$. Let $ R= \{\rho_0, \dots, \rho_s\} $ be the irreducible representations of $ \Gamma $, with $ \rho_0 $ the trivial representation. Let $L$ denote the tautological two-dimensional representation of $ \Gamma $, given by the inclusion $\Gamma\subset \SL(2,\C)$.

The \emph{McKay graph} of $\Gamma$ has vertex set $\{0, 1, \dots, s\}$ where vertex $i$ corresponds to the representation $\rho_i$ of $\Gamma$, with $\dim\Hom_\Gamma(\rho_j,\rho_i\otimes L)$ edges between vertices $i$ and $j$. Note that, as $ L $ is a self-dual $ \Gamma $-representation, this is symmetric in $ i $ and $ j $.

By the McKay correspondence \cite{McKay80}, the McKay graph is an extended Dynkin diagram of $ADE$ type. 
Let $ Q_{\Gamma,0} $ be the vertex set of this graph, and let $ Q_{\Gamma, 1} $ be the set of all pairs consisting of an edge together with an orientation, we shall think of such a pair as an arrow between two vertices in $  Q_{\Gamma,0} $. If $ a\in Q_{\Gamma, 1} $, we write $ \ol{a} $ for the same edge with the opposite orientation.
Then $ Q_\Gamma \coloneqq (Q_{\Gamma,0}, Q_{\Gamma, 1})$ is a quiver, which we will call the \emph{McKay quiver}.
The McKay graph determines an affine Lie algebra $\mathfrak{g}$ (see \eg \cite[Chapter 1.3]{Kac}), and we identify its root lattice with $ \Z^{Q_{\Gamma, 0}} $, such that the  minimal imaginary root of $ \mathfrak{g} $ is identified with \[\delta\coloneqq (\dim \rho_i)_{i\in {Q_{\Gamma,0}}} .\]

We now fix a positive integer $ r $.
Extend $ Q_{\Gamma} $ by introducing a \emph{framing vertex} $ \infty $, together with $ r $ arrows $ b_1,\dots b_r $ from $ \infty $ to $ 0 $, and $ r $ arrows $ \ol{b_1}, \dots, \ol{b_r} $ from $ 0 $ to $ \infty $. Set $ {Q_0} = Q_{\Gamma, 0}\cup \{\infty\}$, and $ Q_1 = Q_{\Gamma, 1}\cup \{b_1,\dots, b_r, \ol{b_1},\dots, \ol{b_r}\}$. Then \[ {Q} \coloneqq ({Q_0}, {Q_1}) \] forms a quiver, which we will call the \emph{framed McKay quiver}.
For each oriented edge $a\in Q_1$ we write $ \tail(a), \head(a) $ for the tail and head of $a$ respectively.

Choose a map $ \epsilon\colon Q_1\to\{-1, 1\}$ such that $ \epsilon(a)\ne \epsilon(\overline{a}) $ for all $ a\in {Q_1} $, and choose a vector $ \mathbf{v}\in \Z^{Q_{\Gamma,0}} $, with all elements nonnegative. 
Furthermore, let $ \rho_\infty $ be a formal symbol such that $ \{\rho_i|i\in Q_0\} $ is a basis for $ \Z^{Q_0} $ as a $ \Z $-module.
Then we can consider the vector $v = \rho_\infty + \mathbf{v}=(1,\mathbf v)\in \Z^{Q_0} $. 
Let \[ M(v) = \bigoplus_{a\in Q_1}\Hom(\C^{v_{\tail(a)}}, \C^{v_{\head(a)}}) \] be the space of all $ Q $-representations of dimension $ v $. The group $ G(v) \coloneqq \Pi_{i\in Q_{0}} \GL(v_i) $ acts naturally on this space by conjugation, with the group $ \C^\times $ of diagonal scalar matrices acting trivially. The induced action of the quotient $ G_{\mathbf v} \coloneqq G(\mathbf v)/\C^\times$ induces a moment map $ \mu $ from $ M(v) $ to the dual $ \mathfrak{g}\dual $ of the Lie algebra $ \mathfrak{g} $ of $ G_{\mathbf v} $ (see \eg \cite[\textsection 8.1]{MFK}).
If we identify $ \mathfrak{g} $ with $ \mathfrak{g}\dual $ using the trace pairing, we find (\cite[676]{Kuznetsov07}) that $ \mu $ is given by
\[  (B)\mapsto \left(\sum_{\substack{\tail(a)=i\\a\in Q_1}}\epsilon(a)B_aB_{\ol{a}}\right)_{i\in Q_0}.  \]

The space of characters of $ G_{ \mathbf v} $ can be identified with integer-valued maps in the space \[ \Theta_{r, \mathbf v}=\{\theta\in \Hom_\Z(\Z^{Q_0}, \Q)\mid\theta(1, \mathbf{v})=0\} \]
by the map \begin{equation}\chi\colon \theta \mapsto \left(g\mapsto \Pi_{i\in Q_0} (\det g_i)^{\theta_i}\right)  .\end{equation}\label{eq:defChi}
We call $ \Theta_{r, \mathbf v} $ the space of stability parameters. As a vector space, it is independent of $ r $, but we indicate $ r $ in the notation because we will later equip $ \Theta_{r, \mathbf v} $ with a wall-and-chamber structure that may depend on $ r $. 

\begin{definition}{(\cite{Nakajima94, Kuznetsov07})}\label{def:quiverVariety}
	Given a stability parameter $ \theta\in \Theta_{r, \mathbf{v}} $, define the \emph{Nakajima quiver variety} to be the GIT quotient \[ \mathfrak{M}_\theta(r, \mathbf{v}) \coloneqq\left({\mu^{-1}(0)\gitquot_{\chi(\theta)}} G_\mathbf{v}\right)_\mathrm{red}= \left(\Proj \bigoplus_{k\ge0}\C[\mu^{-1}(0)]^{k\chi(\theta)}\right)_\mathrm{red}\]

\end{definition}
Note in particular that we give $\mathfrak{M}_{\theta}(r,\mathbf v)$ the reduced scheme structure.
\begin{remarks}
	\begin{itemize}
		\item We reverse the order of the dimension vectors compared to the standard notation appearing in (for example) \cite{Kuznetsov07, VaVa} --- we put the framing rank $ r $ first.		
		\item Some authors replace the exponent $ \theta_i $ in \eqref{eq:defChi} by $ -\theta_i $, and make corresponding sign changes later. This happens, for instance, in \cite{NakALE}.
		\item It is possible to replace $ \mu^{-1}(0) $ appearing in \cref{def:quiverVariety} by $ \mu^{-1}(\lambda) $ for certain $ 0\ne \lambda \in \mathfrak{g}$. This leads, for instance, to moduli spaces of sheaves on \emph{noncommutative} projective planes, see \cite{Kuznetsovetal}.
	\end{itemize}
\end{remarks}

We will mainly focus on the case $ \mathbf{v}=n\delta $ for some integer $ n $.

The structure of these Nakajima quiver varieties is in many ways well understood. We have for instance the following very useful result:
\begin{lemma}
	\label{lem:Mthetageometry}
	For all $\theta\in \Theta_\mathbf{v}$, the scheme $\mathfrak{M}_\theta(r, \mathbf v)$ is irreducible and normal, with symplectic singularities.
\end{lemma}
\begin{proof}
	See Bellamy and Schedler~\cite[Theorem 1.2, Proposition~3.21]{BelSche}.
\end{proof}

\subsection{The preprojective algebra}\label{sec:StabParams}

Let $R(\Gamma)$ denote the representation ring of $\Gamma$. Then $ \Z^{Q_0} \cong\Z\oplus R(\Gamma)$ considered as $\Z$-modules.
Again, let $ v=\rho_\infty + \mathbf{v} \in \Z^{Q_0} $ be nonnegative, and choose some $\theta\in \Theta_{r, \mathbf v}$. 
Let $ \C{Q} $ be the path algebra of $ {Q} $, and
let $ \mathcal{J} $ be the $ \C{Q} $-ideal generated by the expression \begin{equation} \sum_{a\in {Q_1}}\epsilon(a)a\ol{a} .\end{equation}\label{eqn:preprojectiverelation} We will call $ \mathcal{J} $ the (ideal of) \emph{preprojective relations}. 
\begin{definition}\label{def:PreProjAlg}
	Set $ \Pi \coloneqq \C{Q}/\mathcal{J}$. This is the \emph{preprojective algebra} determined by $ Q $.
\end{definition} 
The preprojective algebra $\Pi$ does not depend on the choice of the map~$\epsilon$ \cite[Lemma 2.2]{CBH98}. 

Equivalently, multiplying both sides of this relation by the vertex idempotent of $ \C Q $ at vertex~$i$ shows that~$\Pi$ can be presented as the quotient of $\C Q$ by the ideal
\begin{equation}
	\label{eqn:Pirelations}
	\left( \sum_{\head(a) = i} \epsilon(a) a\ol{a} \mid i\in Q_0\right).
\end{equation}

A $ \Pi $-module is then, up to isomorphism, equivalent to the data of a $ Q $-representation on which the preprojective relations $ \mathcal J $ act by $ 0 $, and we will call such a representation a $ (Q, \mathcal J) $-representation.
Given a $ \Pi $-module $ M $, we write $ M_i $ for the vector space which is the part of $ M $ supported at the vertex $ i $. By the dimension of a $ \Pi $-module $ M $, we shall always mean the dimension \emph{vector} $ (\dim_\infty M, \dim_0 M, \dots, \dim_s M) =(\dim M_\infty, \dim M_0, \dots, \dim M_s)$ of $ M $, and not its dimension as a vector space. For any $ \theta\in \Theta_{r, \mathbf{v}} $ and any $ \Pi $-module $ M $ (or $ Q $-representation $ V $), we will write $ \theta(M), \theta(V) $ for $ \theta(\dim M), \theta(\dim V) $ respectively.

Given a stability parameter $ \theta $, a $ \Pi $-module $ M $ is $ \theta $-\emph{stable}, respectively $ \theta $-\emph{semistable}, if there is no submodule $ U\subset M $ such that $ \theta(U)\le 0 $, respectively $ \theta(U)<0 $. A \emph{$ \theta $-polystable} module is a direct sum of $ \theta $-stable modules.
Two $ \theta $-semistable modules $ M, N $ are said to be \emph{$ S_{\theta} $-equivalent}, if there are composition series of $ \theta $-semistable modules 
\[ 0= M^0\subsetneq M^1\subsetneq\dots\subsetneq M^{k_1}=M, \qquad 0 = N^0\subsetneq N^1\subsetneq \dots \subsetneq N^{k_2}=N \] such that 
\begin{enumerate}
	\item Every $ M^i/M^{i-1} $ and every $ N^j/N^{j-1} $ is a $ \theta $-stable  $ \Pi $-module, and
	\item The polystable modules $ \bigoplus_{i=1}^{k_1}M^i/M^{i-1} $ and $ \bigoplus_{j=1}^{k_2} N^j/N^{j-1} $ are isomorphic.
\end{enumerate}
A one-dimensional $ \Pi $-module will be called a \emph{vertex simple} module.

It will be useful to extend $ S_\theta$-equivalence to $ \Pi $-modules of different dimensions. 

\begin{definition}\label{def:Tequivalence}
	Let $ I\subset {Q_0\setminus \{\infty\}}= Q_{\Gamma, 0} $, and suppose that $ \theta $ is a stability parameter such that $ \theta(\rho_i)=0 $ if and only if $ i\in I $.
	Now let $ M, N $ be two $ \theta $-semistable $ \Pi $-modules, not necessarily of the same dimension.
	Let $ \tilde{M}, \tilde{N} $ be the $ \theta $-polystable modules respectively $ S_{\theta} $-equivalent to $ M$ and $ N $.
	
	We define $ M$ and $ N $ to be \emph{$ R_{\theta} $-equivalent} if there are $ \Pi $-modules $ M', N' $  with $ M'_j = N'_j =0 $ if $ j\not\in I $, such that $ M', N' $ are (possibly empty) direct sums of vertex simple $ \Pi $-modules, and that there is an isomorphism
	\[ \tilde{M}\oplus M' \simeq \tilde{N}\oplus N'.\]
	
	We say that $ \theta \in \Theta_{r, \mathbf v}$ is \emph{generic} if every $ \theta $-semistable $ \Pi $-module of dimension $ (1, \mathbf{v}) $ is $ \theta $-stable.
\end{definition}
It follows that if two modules of the same dimension vector are $ R_{\theta} $-equivalent, they are $ S_{\theta} $-equivalent.
\begin{remark}
	Usually, the notions of (semi)stability and $ S $-equivalence are used without direct reference to the stability condition. But since we will be considering more than one stability condition, we use notation keeping track of the stability conditions where confusion is possible.
\end{remark}

The next result provides an alternative interpretation of quiver varieties, which we will be using throughout. (Recall that $ r $ is part of our definition of the quiver $ Q $.)
\begin{lemma}[{\cite[Prop. 8]{Kuznetsov07}, \cite[Proposition~5.3]{King94}}]\label{lem:quiverVarsAsModuliSpaces}
	There is an isomorphism
	\[\mathfrak{M}_\theta(r,\mathbf{v})\cong \mathcal{M}od_{\Pi}((1, \mathbf{v}), \theta)  \]
	where the right-hand space parametrises $ S_\theta $-equivalence classes of  $ \theta $-semistable $ \Pi $-modules of dimension $ (1, \mathbf{v}) $. As such it is a coarse moduli space of $ \theta $-semistable $ \Pi $-modules of dimension $ (1, \mathbf{v}) $. If $ \theta $ is generic, it is a {fine} moduli space.
\end{lemma}
This viewpoint was introduced in Kronheimer and Nakajima~\cite{KN90} and studied further by Nakajima~\cite{Nakajima94}.

We will make the following slight abuse of terminology: When writing statements like "a $ \theta $-stable $ \Pi $-module $ M\in \mathfrak{M}_{\theta}(r, \mathbf v) $," we mean that $ M $ is a stable $ \Pi $-module representing the isomorphism class of $ \Pi $-modules corresponding to a point $ x\in \mathfrak{M}_\theta(r, \mathbf v)(\C) $.
\subsection{Another construction of Nakajima quiver varieties}\label{sec:VaVaViewpoint}
There is another construction of the varieties $ \mathfrak{M}_\theta(r, \mathbf{v}) $, which we will make use of in \cref{sec:VaVa}. We follow the discussion in \cite{VaVa}.
Instead of constructing the quiver $ Q $, we could also construct a quiver $ Q' = (Q'_0, Q'_1) $, consisting of $ Q_{\Gamma} $ together with the framing vertex $ \infty $ and just two additional arrows: one $a\colon \infty \to 0$ and another $ \ol{a}\colon 0\to \infty $. Then the data of a representation $ R' $ of $ Q' $ such that $ \dim R'_\infty = r $ is equivalent to that of a representation $ R $ of $ Q $ such that $ \dim r_\infty = 1 $. 

The McKay quiver $ Q_\Gamma $ is constructed by taking for each vertex an irreducible representation of $ \Gamma $, and an arrow $ \rho_i\to \rho_j $ if $ \rho_j\subset L\otimes \rho_i $. This implies that the data of an arbitrary $ \Gamma $-representation $ V $ together with a $ \Gamma $-equivariant map $ B\colon V\to V\otimes L $ is equivalent to that of a $ Q_\Gamma $-representation. 
In particular, choose an $ r $-dimensional vector space $ W $ on which $ \Gamma $ acts trivially, and let $ V = \bigoplus_{i\in Q_{\Gamma, 0}} \C^{\mathbf{v}_i}\otimes \rho_i $.
Then the data of an element  $ x\in \mu^{-1}(0)\subset M(v)= M(\rho_\infty + \mathbf v) $ can be written as a tuple \[x = (i_x,j_x,B_x) ,\]
where \[ i_x\colon W\to \bigwedge^2L\otimes V,\quad j_x\colon V\to W,\quad B_x\colon V\to L\otimes V \] are $ \Gamma $-equivariant maps. 

If we let $ M(r, \mathbf{v}) $ be the space of such tuples for a fixed dimension vector $ (r, \mathbf{v}) $, the moment map $ \mu $ can be replaced with the map \[ \mu'\colon M'(v)\to \Hom(V, \wedge^2 L\otimes W) ,\quad (i_x, j_x, B_x)\mapsto (B_x\wedge B_x-i_x\circ j_x).\]
As before, $ M'((r, \mathbf{v})) $ carries a natural action by the group $ G_{\mathbf v} $, and the quiver variety $ \mathfrak{M}_\theta(r,\mathbf{v}) $ can also be defined as the GIT quotient of $ \mu'^{-1}(0) $ by this action:
\[ \mathfrak{M}_\theta(r,\mathbf{v})= \mu'^{-1}(0)\gitquot_{\chi(\theta)}G_\mathbf{v}.\]

Given a $ Q' $-representation $ M $, $M\in \mu^{-1}(0) $ if and only if a similar ideal of preprojective relations acts by $ 0 $ on $ M $. To reduce notational clutter, we shall also write this ideal as $ \mathcal J $, and similarly speak of $ (Q',\mathcal J) $-representations.
Since we shall have to work with $ Q' $ -representations of different dimension vectors, it will be convenient to refer explicitly to the $ \Gamma $-representations $ V, W $. For this reason, we shall write $ x = (V_x, W_x, i_x, j_x, B_x) $. 

The construction here described was the original definition of Nakajima quiver varieties (\cite{Nakajima94}). Our definition using $ 2r $ framing arrows to set $ \dim M_\infty =1$ for a $ \Pi $-module $ M $ is usually called the \emph{Crawley-Boevey trick}, first explicitly noted in \cite{CrawBoe}.

\subsection{Stability parameters}\label{sec:stabParams}\label{subsec:wallandcham}
We now go back to the description of $ \mathfrak{M}_\theta(r, \mathbf v) $ from \cref{sec:StabParams} as a moduli space for $ \theta $-semistable $ \Pi $-modules.

The set of stability conditions $\Theta_{r, \mathbf v}$ admits a preorder $\geq$, where $\theta\geq \theta'$ if and only if every $\theta$-semistable $\Pi$-module is $\theta'$-semistable. This determines a wall-and-chamber structure (\cite{DH98, Thaddeus96}), where
$\theta,\theta' \in \Theta_{r, \mathbf v}$ lie in the relative interior of the same cone if and only if both $\theta \geq \theta'$ and $\theta' \geq\theta$, implying that $\mathfrak{M}_{\theta}(r,\mathbf v)\cong\mathfrak{M}_{\theta'}(r,\mathbf v)$. The interiors of the top-dimensional cones in $\Theta_{r, \mathbf v}$ are GIT chambers, and the codimension-one faces of the closure of a GIT chamber are GIT walls. If a stability parameter $\theta \in \Theta_{r, \mathbf v}$ lies in a GIT chamber -- equivalently, if any $ \theta $-semistable $ \Pi $-module of dimension $ (1, \mathbf v) $ is $ \theta $-stable -- we say that $ \theta $ is \emph{generic}.

The vector $(1,\mathbf v)$ is indivisible, so \cite[Proposition~5.3]{King94} implies that for any generic $\theta \in \Theta_{r, \mathbf v}$ the quiver variety $\mathfrak{M}_{\theta}(r,\mathbf v)$ is the fine moduli space of isomorphism classes of $\theta$-stable $\Pi$-modules of dimension vector $(1,\mathbf v)$.

Let us write \[ C^+_{r,\mathbf{v}} = \{\theta\in \Theta_{r, \mathbf v}\mid\theta(\rho_i)>0\textrm{ for all $ i $.}\} \]  This may not in general be a chamber in the wall-and-chamber structure on $ \Theta_{r, \mathbf v} $, but it is contained in one, and if $ \mathbf{v}= n\delta $, it is a genuine chamber -- see \cite[Example~2.1]{BelCra} or \cite[Section 2.8]{Nak_branching} for the case $ r=1 $, and \cite[Theorem 4.15]{BellamyCrawSchedler} for the general case. 

\emph{From now on}, for any choice of $ r$ and  $ \mathbf v $, we will denote by $ \theta $  any stability parameter in the cone $ C^+_{r,\mathbf{v}} $. We will use the phrases `stable module' and `$ \theta $-stable module' equivalently. We will often leave the framing rank and dimension vector implicit and simply write $ C^+ = C^+_{r, \mathbf v} $.

\begin{remark}\label{rmk: stabilityForGreps}
	We can translate this into a notion of stability for $ Q' $-representations. Let $M = (W, V, i, j, B)  \in \mu^{-1}(0) $ be a $ Q' $-representation. Then $ M $ is stable if there is no strict subspace $ U\subsetneq V$ such that $ i(W)\subset U $ and $ B(U)\subset L\otimes U $.
\end{remark}

In addition to the chamber $ C^+ $, we shall focus on a particular stability parameter in is boundary, namely \[ \theta_0 = (-n, 1, 0, 0 \dots, 0) ,\]
i.e., $ \theta_0(\rho_\infty) = -n,\ \theta_0(\rho_0)=1,\ \theta_0(\rho_i)=0 $ for other $ i $, where $ n=\mathbf v_0 $.
We fix this parameter for the rest of the paper.

Whether $ C^{+}_{r, \mathbf v} $ is a chamber or not, there is always a morphism $ \mathfrak{M}_\theta(r, \mathbf{v})\to \mathfrak{M}_{\theta_0}(r, \mathbf{v}) $ induced by variation of GIT parameter. 
We shall need some results about it in a special case:

\begin{lemma}\label{lem:VGITsurjectivity}
Let $ r, n\in \N $. Then 
\begin{enumerate}
	
	\item $\dim \mathfrak{M}_{\theta}(r,n\delta)=\dim \mathfrak M_{\theta_0}(r, n\delta)=2rn $,
	\item 
	the projective morphism $ \mathfrak{M}_{\theta}(r,n\delta)\to  \mathfrak M_{\theta_0}(r, n\delta)  $
	induced by varying the parameter $ \theta $ to $ \theta_0 $ is a resolution of singularities,
\end{enumerate} 
\end{lemma}
We postpone the complete proof to \cref{sec:surjecProof}, as we shall need explicit interpretations of the quiver varieties for it. There we also give a more primitive proof that $ C^+_{r, n\delta} $ is a chamber.

\subsection{The concentrated module}

We now introduce the \emph{concentrated module} of an $R_{\theta_0} $-equivalence class, which we will see uniquely characterises the class.

Let $ [M] $ be an $ S $-equivalence class of $ \theta_0 $-semistable $ \Pi $-modules of dimension vector $ (1,\mathbf{v}) $.

Then $ [M] $ has a $ \theta_0 $-polystable member $ \tilde{M} $, which is unique up to isomorphism.
\begin{definition}\label{def:concentratedModule}
The \emph{concentrated module (associated to $ [M] $)} is the unique summand $ M\con $ of $ \tilde{M} $ such that $ \dim_\infty M\con =1 $.
\end{definition}

\begin{corollary}\label{cor:dimConcentratedModule}
We have $ \dim_0 M\con= \mathbf v_0 $.
\end{corollary}
\begin{proof}
Since every summand $ N $ of $ \tilde{M} $ satisfies $ \theta_0(N)=0 $, the same must hold for $ M\con $. Since $ \dim_\infty M\con = 1 $, $ \theta_0(\dim M\con)=0 $, and $ \theta_{0}(\rho_i)>0 $ precisely for $ i=0$, we must have $ \dim_0 M\con= \mathbf{v}_0$.
\end{proof}

\begin{remark}\label{rem:vertex_simples}
In a $ \theta_0 $-polystable module, all summands other than the concentrated module are 1-dimensional vertex simples. To see this, note that removing the vertices $\infty, 0$ and their incident edges from the framed McKay graph leaves us with an finite-type Dynkin diagram. 

Then, since a stable summand $N\ne M\con $ must satisfy $ \dim_i N=0 $ for $ i\in{\infty, 0} $, we can identify it with a simple module of the preprojective algebra of a quiver of finite type. But such modules are one-dimensional by \cite[Lemma~2.2]{SavTing}.
\end{remark}

We move from $ S $-equivalence to $ R $-equivalence:
\begin{lemma}
Let $ M,N $ be two $ R_{\theta_0} $-equivalent $ \Pi $-modules. The concentrated modules of their $ S_{\theta_0} $-equivalence classes are isomorphic.
\end{lemma}

\begin{proof}
This immediately follows from \cref{def:Tequivalence} and \cref{rem:vertex_simples}.

\end{proof}

It follows that every $ R_{\theta_0} $-equivalence class is characterised by a concentrated module, which is unique up to isomorphism. With a small abuse of language, we will often speak of \emph{the} concentrated module. 
\begin{corollary}
The concentrated module of an $ R_{\theta_0} $-equivalence class is both stable and $ \theta_0 $-stable.
\end{corollary}
\begin{proof}
Clear.
\end{proof}

A reason for introducing the notion of concentrated module is the following result:

\begin{lemma}\label{lem: concentrated is a quotient}
Let $ M $ be a $ \theta $-stable $ \Pi $-module. Consider $ M $ as a $ \theta_0 $-semistable module, and let $ M\con $ be the concentrated representation of its $ R_{\theta_0} $-equivalence class.
Then $ M\con $ is a quotient of $ M $.
\end{lemma}

\begin{proof}
Note that $ M\con $ is, by definition, a subquotient of $ M $. If $ M\con $ is a quotient of a strict submodule $ N\subsetneq M $, we must have $ \dim_\infty N = \dim_\infty M=\dim_\infty M\con =1 $. But then $ \theta(N)<0 $, contradicting stability of $ M $.
\end{proof}

\subsection{Stacks, quotients and descent}
In this subsection, we collect the necessary background for working with the stack compactification of $ \C^2/\Gamma $.

\begin{definition}[{\cite[Definition 11.1.1]{Olsson}}]\label{def:CoarseModSpace}
	Given a Deligne-Mumford stack $ \mathcal{X} $, a \emph{coarse moduli space} for $ \mathcal{X} $ is a morphism  $t\colon \mathcal X\to X $, where $ X $ is an algebraic space, such that 
	\begin{itemize}
		\item any other morphism from $ \mathcal X $ to an algebraic space factors through $ t $,
		\item for every algebraically closed field $ k $, the map of sets $ |\mathcal{X}|(k)\to X(k) $ induced by $ t $ is bijective, where $ |\mathcal{X}| $ is the set of isomorphism classes of morphisms $ \Spec k\to \mathcal X $.
	\end{itemize} 
If the coarse moduli space is a scheme, we call it a coarse moduli scheme.
\end{definition}

\begin{lemma}[{\cite[Proposition 11.3.4]{Olsson}}]\label{lem:TameStackCohom}
	Let $ \mathcal X $ be a tame locally finite Deligne-Mumford stack with finite diagonal, and with a coarse moduli space $f\colon \mathcal X\to X $. Then the functor $ f_* $ is exact. 
\end{lemma}

In some cases scheme-theoretic quotients by finite groups agree with the stack-theoretic quotients. In these cases, we can give an alternative description of how sheaves descend:

\begin{proposition}\label{thm:descentWhenScheme}
	Let $ G $ be a finite group acting on a scheme $ X $, such that the orbit of every closed point is contained in an open affine subscheme\footnote{Note that there are schemes acted on by finite groups such that there is at least one orbit of a closed point not contained in an affine subscheme: In the line with double origin, with an action of $ \Z/2\Z $ interchanging the origins, the union of the two origins does not lie in an affine subscheme. For an example of a complete variety also satisfying this property, see \cite{Speyer}.}. 
	
	Then:
	\begin{enumerate}
		\item There is a finite $ G $-invariant morphism $q\colon X\to X/G $ which is a categorical quotient, i.e. with the universal property that any $ G $-invariant morphism $ X\to Y $ of schemes factors uniquely through $ X/G $.
		The structure sheaf of $ X/G $ satisfies \[ \OO_{X/G}(U)=(\OO_{X}(U))^G \] for any open $ U\subset X/G $.
		
		\item If we also assume that the action of $ G $ is free, the stack quotient $[X/G] $ is represented by $ X/G $. Furthermore, $ q $ is a flat morphism of degree $ |G| $, and the functor \[ q_*(-)^G\colon  \ECoh{G}{X}\to \Coh{X/G}  \] defined by \[ q_*(\shf F)^G(U) = (q_*(\shf F)(U))^G = (\shf F(q^{-1}(U)))^G\] for any open $ U\subset X/G $ is an equivalence. 
		The inverse of $ q_*(-)^G $ is given by $ q^* $.
	\end{enumerate}
\end{proposition}
\begin{proof}
	This is a special case of \cite[Theorem 1, p. 111]{Mumford}.
\end{proof}

We will also need to know how sheaf cohomology behaves under descent along group quotients:
\begin{lemma}\label{lem:specseqforGammaStack}
	Let $ X $ be a scheme, acted on by a finite group $ G $. Let $ \shf F $ be a $ G $-equivariant sheaf on $ X $.
	Write $ \shf F' $ for the sheaf on $ [X/G] $ obtained from $ \shf F $ by descent. Then there is a spectral sequence
	\[ E^{p,q}_2 = H^p(G, H^q(X, \shf F))\Rightarrow H^{p+q}([X/G], \shf F').\]

\end{lemma}
\begin{proof}
	See \eg \cite[Theorem A.6]{FHRZ08}.\footnote{Their statement is in a slightly different context, but the proof works verbatim here.}
\end{proof}
For us, since we work over a field of characteristic zero, the functor $ M\mapsto M^G $ will always be exact, and we have $ H^p([X/G], \shf F') = H^p(X, \shf F)^G $.

\begin{remark}
	It is	 common to use the same notation for $ \shf F $ and $ \shf F' $, as descent induces an equivalence of categories $ \ECoh{G}{X}\cong \Coh{[X/G]} $. We will make the distinction in order to (hopefully) reduce confusion.
\end{remark}

\subsection{Framed sheaves on projective Deligne-Mumford stacks}

See \cite[Corollary 5.4, Definition 5.5]{Kresch} for the definition of a \emph{projective} Deligne-Mumford stack.

Let $ \mathcal Y $ be a normal irreducible projective Deligne-Mumford stack, with a coarse moduli scheme $ \tau\colon \mathcal Y \to Y $.

\begin{definition}[{\cite[Definitions 3.1, 3.3]{BruzzoSala}}]\label{def:FramedSheaf}
	Let $ \shf F $ be a quasi-coherent sheaf on $ \mathcal Y $. An \emph{$ \shf F $-framed} sheaf on $ \mathcal Y $ is a pair $ (\shf G, \phi_{\shf G}) $, where $ \shf G $ is a coherent sheaf on $ \mathcal Y $, and $ \phi_{\shf G}\colon \shf G\to \shf F $ is a sheaf homomorphism.
	A morphism of $ \shf F $-framed sheaves $ (\shf G, \phi_{\shf G})\to (\shf H, \phi_{\shf H}) $ consists of a sheaf morphism $ f\colon \shf G\to\shf H $ such that there is a $c\in \C^\times $ with $ \phi_{\shf G}\circ f=c\phi_{\shf H} $.
A morphism $  f\colon (\shf G, \phi_{\shf G})\to (\shf H, \phi_{\shf H}) $ of $ \shf F $-framed sheaves is injective, respectively surjective, respectively an isomorphism, if $  f\colon \shf G\to \shf H $ is. Clearly, if $ f $ is an isomorphism, we can take $ c=1 $.
\end{definition}

\begin{definition}\label{def:FramedAlongDivisor}
	Let $ \mathcal D\subset \mathcal Y $ be a smooth integral closed substack of codimension 1.
	A coherent sheaf on $ \mathcal Y $ \emph{framed along} $ \mathcal D $, or just a \emph{framed sheaf}, is a $ \OO_{\mathcal D}^r $-framed sheaf $ (\shf G, \phi_{\shf G}) $, where $ r $ is the rank of $ \shf G $, and $ \phi_{\shf G}|_{\mathcal D }\colon \shf G|_{\mathcal D }\isoto\OO_{\mathcal D}^r$ is an isomorphism.
\end{definition}
\begin{remark}
	Note that we differ from the definition in \cite[5]{BruzzoSala} on several points:
	\begin{itemize}
		\item We make no requirement on $\dim \mathcal Y $, and we do not require $ \mathcal Y $ to be smooth -- in fact, we will usually work with a singular $ \mathcal Y $;
		\item We do not make any assumptions on the coarse moduli space of $ \mathcal D $;
		\item We do not require $ \shf G $ to be locally free in a neighbourhood of $\mathcal{D}$ -- in our cases, we will see in \cref{lem:locFreeAroundFraming,,lem:locFreeAroundFramingX} that this property follows from the other conditions.
	\end{itemize}
\end{remark}

	\section{Quotients of the projective plane}\label{chap:SheavesOnQuotients}

\subsection{The stack $ \mathcal X $}\label{sec:stackX}\label{sec:weirdQuotient}
We start by constructing the stack $ \mathcal X $.

First, we construct the scheme quotient $ \P^2/\Gamma $ by considering the action of $ \Gamma $ on the homogeneous coordinate ring $ \C[x,y,z] $ of $ \P^2 $. 
It is well-known that the invariant ring $ \C[x,y]^\Gamma $ is generated as a $ \C $-algebra by three homogeneous elements. These three, together with $ z $ will then generate the ring $ \C[x,y,z]^\Gamma $. Weighting these four elements by their degrees, we find \[ \P^2/\Gamma = \Proj \C[x,y,z]^\Gamma ,\]
naturally embedding as a singular surface in a three-dimensional weighted projective space.

Consider the quotient morphism $ \P^2\to \P^2/\Gamma $, and let $ l_\infty = \{z=0\} \subset \P^2$, the `line at infinity'.
The complement of $l_\infty $ is an affine plane, the image of which under this morphism is the Kleinian singularity $ \C^2/\Gamma = \Spec \C[x,y]^\Gamma $. It has a unique singular point, and we write $ o $ for the preimage of this point (the `origin') in $ \P^2 $.

We may also instead consider the stack quotient $ [\P^2/\Gamma] $. There are then surjective morphisms of stacks \[ \P^2\to [\P^2/\Gamma]\to \P^2/\Gamma .\]

Let now $ U_f $ be the open subscheme $ \P^2\setminus(\{o\}\cup l_\infty) $. As $ \Gamma $ acts freely on $ U_f $, \cref{thm:descentWhenScheme} applies, and the stack $ [U_f/\Gamma]$ is represented by a scheme $ U_f/\Gamma $. As a slight abuse of notation, we will simply write $ [U_f/\Gamma] = U_f/\Gamma $ in what follows.
\begin{definition}\label{def:XStack}
	We define the stack $ \mathcal X $ as the pushout in the diagram of Deligne-Mumford stacks
	\[\begin{tikzcd}
		U_f/\Gamma \arrow[d, hook] \arrow[r, hook] &\C^2/\Gamma \arrow[d, dashed] \\
		{[\P^2\setminus\{o\}/\Gamma]} \arrow[r, dashed]              &  \mathcal X
	\end{tikzcd}\]
	
\end{definition}
That it is possible to take pushouts of stacks in this manner, follows for instance from \cite[Theorem C]{Rydh}, which also shows that the two dashed morphisms are open immersions.

Intuitively, $ \mathcal X $ is a quotient of $ \P^2 $ by $ \Gamma $, where we have taken the scheme quotient around $ 0 $, and the stack quotient around $ l_\infty $. 

Set $ d_\infty = [l_\infty/\Gamma] $ and $ r_\infty = l_\infty/\Gamma$. These will differ as $ \Gamma $ does not act freely on $ l_\infty $. Then $ d_\infty $ is a closed substack of both $ [\P^2/\Gamma] $ and $ \mathcal X $.
We thus have the following commutative diagram of stack morphisms, where we name the various morphisms as indicated.

\begin{equation}\label{diag:bigDiagram}
	\begin{tikzcd}
		l_\infty \arrow[dd, two heads, "\pi|_{l_\infty}"'] \arrow[r, hook, "i"]                 & \P^2 \arrow[d, two heads, "k"] \arrow[dd, two heads, bend left=49, "\pi"] \\
		& {[\P^2/\Gamma]} \arrow[d, two heads, "q"]               \\
		d_\infty \arrow[d, two heads] \arrow[r, hook, "j"] \arrow[ru, hook] & \mathcal X \arrow[d, two heads, "c"]                             \\
		r_\infty \arrow[r, hook]                                       & \P^2/\Gamma                                       
	\end{tikzcd}.
\end{equation}
We may note in diagram~\eqref{diag:bigDiagram} that $ k $ is étale. The morphism $ \pi $ is \emph{not} étale, but its restriction to $ \P^2\setminus \{o\} $ (on which we can identify it with $ k $) is. Furthermore, $ c\circ\pi= c\circ q\circ k $ is a finite morphism of schemes, by \cref{thm:descentWhenScheme}.

We shall need some more details on $\mathcal{X}$. 
\begin{lemma}\label{lem:XisDM}
	The stack $ \mathcal X $ is a proper, tame Deligne-Mumford stack of finite type, with finite diagonal.
\end{lemma}
\begin{proof}
	$ \mathcal X $ is Deligne-Mumford because it is the pushout of the Deligne-Mumford stacks $ [(\P^2\setminus \{o\})/\Gamma] $ and $\C^2/\Gamma$ along the Deligne-Mumford stack $ U_f $.
	Explicitly, an étale atlas for $ \mathcal X $ is given by the map $ (\P^2\setminus \{o\})\coprod (\C^2/\Gamma)\to \mathcal X$ induced by the definition of $ \mathcal X $.
	Then, to show that $ \mathcal X $ is proper over $ \C $, we use \cite[Theorem 3.1(2)]{Conrad} to show that $ \mathcal{X}\to \P^2/\Gamma $ is proper.
Now $ \P^2/\Gamma $ is proper, being projective, and since the composition of proper morphisms of stacks is proper (\cite[{Tag 0CL7}]{stacks}), $ \mathcal{X} $ is proper.

To see that $ \mathcal{X} $ is of finite type, recall that being of finite type is equivalent to being locally of finite type and quasicompact. It is simple to check on the patches $ [ \P^2\setminus \{o\}/\Gamma] $ and $ \C^2/\Gamma $ that $ \mathcal{X} $ is locally of finite type. To see that it's quasicompact, apply \cite[Tag 04YC, (3),(5)]{stacks} to the composition of morphisms $ \P^2\to[\P^2/\Gamma]\to \mathcal{X} $.

It is tame: Because $ \mathcal X $ is proper, it is separated. Because $ \mathcal X $ is Deligne-Mumford, every automorphism group of a geometric point is finite, so its order is invertible in $ \C $.

Finally, to determine whether the diagonal morphism $\Delta_{\mathcal X}\colon \mathcal{X}\to \mathcal{X}\times_\C \mathcal X $ is finite, it is enough \cite[Tag 04XC]{stacks} to work étale-locally on the target. Now an étale cover of $  \mathcal X\times \mathcal X $ consists of $ \C^2/\Gamma\times \C^2/\Gamma $, $ \C^2/\Gamma\times[\P^2\setminus \{o\}/\Gamma] $, $ [\P^2\setminus \{o\}/\Gamma]\times \C^2/\Gamma $, and $ [\P^2\setminus \{o\}/\Gamma]\times[\P^2\setminus \{o\}/\Gamma] $ and it is simple to prove that the base change of $ \Delta_{\mathcal X} $
to any of these is finite.
\end{proof}

Similarly, $ [\P^2/\Gamma] $ is also a proper tame Deligne-Mumford stack of finite type. with finite diagonal. Since $ [\P^2/\Gamma] $ is smooth, the singular locus of $ \mathcal{X} $ is precisely the unique singular point of $ \C^2/\Gamma\subset \mathcal X $.

Let us also note that 
\begin{equation}\label{eq:fibreprod}
 \C^2/\Gamma\times_{\mathcal X}[(\P^2\setminus \{o\})/\Gamma] = U_f/\Gamma.
\end{equation}

To see this, we can set up the commutative diagram of fibre products
\[ \begin{tikzcd}
	\C^2/\Gamma\times_{\P^2/\Gamma}(\P^2\setminus \{o\})/\Gamma \arrow[d] \arrow[r] & \C^2/\Gamma \arrow[d] \\
	(\P^2\setminus \{o\})/\Gamma\times_{\P^2/\Gamma}\mathcal{X} \arrow[d] \arrow[r]                & \mathcal X \arrow[d, "c"]  \\
	(\P^2\setminus \{o\})/\Gamma \arrow[r]                              & \P^2/\Gamma          
\end{tikzcd} ,\]
where the top-left downwards morphism is induced by the universal property of $ (\P^2\setminus \{o\})/\Gamma\times_{\P^2/\Gamma}\mathcal{X}  $.

As
\[ (\P^2\setminus \{o\})/\Gamma\times_{\P^2/\Gamma}\mathcal{X}={[(\P^2\setminus \{o\})/\Gamma]} ,\] and \[ \C^2/\Gamma\times_{\P^2/\Gamma}(\P^2\setminus \{o\})/\Gamma = (\C^2\setminus \{o\})/\Gamma =U_f/\Gamma,\] \eqref{eq:fibreprod} holds.

\subsection{Defining a sheaf on $\mathcal X$}

By a sheaf on $ \mathcal X $, we mean a sheaf defined on the étale site of $ \mathcal{X} $ (see \cite[Définition 12.1.(ii)]{LMB} for details).
Thus a sheaf $ \shf F $ on $\mathcal X$ is the data consisting of
\begin{itemize}
	\item for every étale morphism $u\colon U\to \mathcal X$ where $ U $ is a scheme, a (Zariski) sheaf $\shf F_U$ on $U$,
	\item  for any two étale morphisms from schemes $ U_i\to \mathcal X,\  U_j\to \mathcal X $,  an isomorphism $ \psi_{ij}\colon \mathrm{pr}_i^*\shf F_{U_i}\to \mathrm{pr}_j^*\shf F_{U_j} $ over the fibre product $ U_i\times_{\mathcal X} U_j $,
	\item such that the isomorphisms $ \psi_{ij} $ satisfy a \emph{cocycle condition} (see \cite[tag 03O7]{stacks})
\end{itemize} 
Here $ \mathrm{pr}_i, \mathrm{pr}_j $ are the projections from $ U_i\times_{\mathcal X} U_j $ onto the first and second factor.

If $u\colon U\to \mathcal X$ is such an étale map, we can take a Zariski cover $ \{U_1, \dots, U_k\} $ of $U$ such that for every $ i $, $ c\circ u(U_i) $ has empty intersection with either $ \{o\} $ or $ r_\infty $. Then a Zariski sheaf on $ U $ is given by Zariski sheaves on each $ U_i, i\in \{1,\dots, k\} $, that glue together on the overlaps $ U_{i}\cap U_j $.
It follows that we may assume that either $ c\circ u(U_i)\cap \{o\} = \emptyset $ or $ c\circ u(U_i)\cap r_\infty = \emptyset $.  In the first case, $ u $ factors through $ [(\P^2\setminus \{o\})/\Gamma] $, and in the second, through $ \C^2/\Gamma $. 

In conclusion, to define an (étale) sheaf $ \shf F $ on $ \mathcal{X} $, it is enough to define its restrictions  $ \shf F|_{\C^2/\Gamma} $ and $ \shf F|_{[(\P^2\setminus \{o\})/\Gamma]} $, together with a choice of isomorphism between the restrictions of these sheaves to $ U_f/\Gamma $. But to define a sheaf on $ U_f/\Gamma = [(\P^2\setminus (\{o\}\cup l_\infty))/\Gamma] $ is equivalent to defining a $ \Gamma $-equivariant sheaf on $ (\P^2\setminus (\{o\}\cup l_\infty)) $. 

So we can instead define $\shf F $ by choosing a  sheaf $ \shf F_{(\P^2\setminus \{o\})} $ on $ \P^2\setminus \{o\} $ and a sheaf $ \shf F_{\C^2/\Gamma} $ on $ \C^2/\Gamma $ together with a $ \Gamma $-equivariant isomorphism of their inverse images on $ U_f/\Gamma= \C^2/\Gamma\times_{\mathcal{X}}[(\P^2\setminus \{o\}/\Gamma)] $ -- i.e. an isomorphism of $ \Gamma $-equivariant sheaves \[ \iota_{\shf F}\colon(\shf F_{(\P^2\setminus \{o\})})|_{U_f}\isoto  \left((\pi|_{\P^2\setminus l_\infty})^*\shf F_{\C^2/\Gamma}\right)|_{U_f}\] on $ U_f $.

(Since this étale cover of $ \mathcal{X} $ consists of only two schemes, the cocycle condition is trivially satisfied.) All our sheaves on $ \mathcal{X} $ will be described in this fashion.

\begin{example}
	We can describe the ideal sheaf $ \shf I_{d_\infty} $ of the closed substack $ d_\infty\subset \mathcal X $ as follows: 
	
	The ideal sheaf $ \shf I_{l_\infty} $ of $ l_\infty $ in $ \P^2$ is $ \Gamma $-equivariant, locally generated by a section $ z\in \OO_{\P^2} $. Over the subscheme $ U_f, $ it is isomorphic to the structure sheaf through the morphism locally defined by mapping a section $ f $ to $\frac{f}{z} $, since $ z $ is invertible.
	Thus we can take \[ (\shf I_{d_\infty})_{(\C^2/\Gamma)}= \OO_{\C^2/\Gamma}, \quad (\shf I_{d_\infty})_{(\P^2\setminus \{o\})}= \shf I_{l_\infty}|_{(\P^2\setminus \{o\})} ,\quad \iota_{\shf I_{d_\infty}}\colon f\mapsto \frac{f}{z} ,\] which completely defines $ \shf I_{d_\infty} $.
\end{example}

\subsection{Notation}
Let's fix some notation for various substacks of $ \P^2 $ and $ \mathcal{X} $.
As mentioned, we let $ o $ be the "origin" of the affine plane $U_o \coloneqq \P^2\setminus l_\infty $, considered as an open subscheme of $ \P^2 $. Similarly, we set $ U_l \coloneqq \P^2\setminus \{o\}$.
The image of $ o $ under $ \pi $ is the unique singular point of $ \mathcal X $, call it $ p $. 
Similarly, write $ V_d $ for the stack $ [(\P^2\setminus \{o\})/\Gamma] $, considered as an open substack of $ \mathcal X $, and $ V_p $ for the singular affine scheme $ \C^2/\Gamma = (\P^2\setminus l_\infty)/\Gamma $, also considered as an open substack of $ \mathcal X $.

\subsection{Projectivity of $ \mathcal{X} $}
We shall need a few more properties of $ \mathcal{X} $, namely; we must find its coarse moduli space, and we shall show that $ \mathcal X $ is a projective stack.

The coarse moduli space of $ \mathcal{X} $ is as one should expect:
\begin{lemma}\label{lem:CoarseModSpaceOfX}
	The scheme $ \P^2/\Gamma $ is the coarse moduli space for both $ [\P^2/\Gamma] $ and $ \mathcal{X} $.
	Furthermore, the functors \[ c_* \colon \Coh{\mathcal{X}}\to \Coh{\P^2/\Gamma} \textrm{ and } (c\circ q)_*\colon \Coh{[\P^2/\Gamma]}\to \Coh{\P^2/\Gamma} \] are exact.
\end{lemma}
\begin{proof}
	By \cite[Theorem 3.1]{Conrad}), the coarse moduli space of $ [\P^2/\Gamma] $ must be a scheme.  As any $ \Gamma $-invariant morphism of schemes $ \P^2\to Y $ factors through $ \P^2/\Gamma $ (\cref{thm:descentWhenScheme}), this shows that any morphism from $ [\P^2/\Gamma] $ to a scheme factors through $ \P^2/\Gamma $. 
	The second condition of \cref{def:CoarseModSpace}, that $ c $ induces a bijection $ [\P^2/\Gamma](k)\isoto \P^2/\Gamma(k) $ for any algebraically closed field $ k $ is clearly satisfied.
	Since the map $ [\P^2/\Gamma] \to \P^2/\Gamma$ factors through $ \mathcal X $, it follows that $ \P^2/\Gamma $ is also the coarse moduli space of $ \mathcal X $.
	
	The final assertion follows from the tameness of $ \mathcal{X} $ and \cref{lem:TameStackCohom}.
\end{proof}

\begin{proposition}\label{prop:XIsProjective}
	The stack $ \mathcal X $ is projective, and $\OO_{\mathcal X}(d_\infty)$ is a generating sheaf for $ \mathcal{X} $.
\end{proposition}
\begin{proof}
	
	It is clear that the coarse moduli space of $ \mathcal X $, namely $ \P^2/\Gamma $, is a projective scheme.

	We claim that the sheaf $\OO_{\mathcal X}(d_\infty) $ is a generating sheaf for $ \mathcal X $.
	So let $ x\colon p=\Spec \C\to \mathcal X $ be a point, and assume that it has nontrivial stabiliser. Then $ \im x\subset d_\infty $.
	Let $ \{p_1, \dots, p_k\}\subset l_\infty $ be the $ \Gamma $-orbit in $ \P^2\setminus \{o\} $ corresponding to $ p $. There is then, for any $ p_i $, a commutative diagram
	
	\[
	\begin{tikzcd}
		p_i \arrow[d] \arrow[r, "x_i"] & \P^2\setminus \{o\} = U_l \arrow[d, "\pi"]  &\\
		p \arrow[r, "x"]               & {[\P^2\setminus \{o\}/\Gamma]}=V_d\arrow[hook, r]& \mathcal X\textrm{.}
	\end{tikzcd}\]
Let $ \Gamma_{x_i} $ be the stabiliser of $ x_i $.	We see that $ \pi^*\OO_{\mathcal X}(d_\infty)|_{V_d}=\OO_{\P^2}(l_\infty)|_{U_l} $.

If we think of $ l_\infty $ as the projectivisation of $ \C^2 $, we can then identify the fiber $ x_i^*\OO_{\mathcal X}(d_\infty) $ with the affine line $ l_i \subset \C^2 $ corresponding to $ x_i $.
	Now, since $ \Gamma $ acts freely on $ \C^2\setminus \{o\} $,  $ \Gamma_{x_i} $ must act freely on $ l_i\setminus 0 $. Explicitly, if we denote the homogeneous prime ideal of $ \C[x,y] $ associated to $ p_i $ by $ \mathfrak{m}_{p_i} $, the line corresponding to $ p_i $ can then be $ \Gamma_{x_i} $-invariantly identified with the 1-dimensional $ \C $-vector space
	\[ x_i^*\OO_{\P^2}(l_\infty) = \left(\frac{1}{z}\C[x,y,z]/(\mathfrak{m}_{p_i})\right)_0. \]  It follows that the action of $ \Gamma_{p_i} $ on this line is faithful.
	
	Finally, let $ x\colon \Spec k\to \mathcal X $ be another morphism with $ k $ an algebraically closed field. Then $ x $ either factors through $ V_p $ or $ V_d $. In the first case the stabiliser is trivial, and in the second the stabiliser is either trivial, or $ x $ factors through one of the morphisms $ \Spec \C\to \mathcal X $ considered above.	
	Then the action of the stabiliser groups on the fibres of $ \OO_{\mathcal X}(d_\infty) $ is faithful, and $ \mathcal{X} $ is projective.
	
\end{proof}

\section{Framed sheaves on the projective plane and on the stack}\label{sec:moving up and down}

In this section, we show how to move between the category of coherent $ \Gamma $-equivariant framed sheaves on $ \P^2 $ and the category of coherent framed sheaves on $ \mathcal X $. 

\subsection{Sets of framed sheaves}\label{sec:setsOfSheaves}
Associated to the morphism $ \pi\colon \P^2\to \mathcal X $, there are functors $ \pi^* $ and $ \pi_* $. We will, however, need to modify them both. As we will be working with $ \Gamma $-equivariant sheaves on $ \P^2 $, we will consider the $\Gamma $-invariants of the direct image by $ \pi $, while taking into account the glued-together nature of $\mathcal X $.

More explicitly, we do as follows:
\begin{definition}\label{def:functorD}
	
	Let $ \shf E $ be a $ \Gamma $-equivariant coherent sheaf on $ \P^2 $. By descent, there is a sheaf $ \shf F $ on $ V_d $ such that $ (\pi|_{U_l})^*\shf F = \shf E|_{U_l} $.
	As for $ V_p $, restrict $ \shf E $ to $ U_o =\C^2 $, and form the sheaf $ (\pi|_{U_o})_*(\shf E|_{U_o})^\Gamma $ on $ V_p $. 
	By \cref{thm:descentWhenScheme}, we have \[ (\pi|_{U_o})_*(\shf E|_{U_o})^\Gamma|_{U_f/\Gamma}=\shf F|_{U_f/\Gamma}\] on $ U_f/\Gamma = [U_f/\Gamma] $, so we can glue these sheaves together to a sheaf $ D(\shf E) $ on $ \mathcal{X} $.
	
	This defines a functor $ D\colon \ECoh{\Gamma}{\P^2}\to \Coh{\mathcal X} $.
\end{definition}
\begin{remark}
	With the chain of morphisms $ \P^2\to [\P^2/\Gamma]\to \mathcal X $, the functor $ D $ is just the descent along the first morphism followed by the direct image along the second morphism. But the explicit description in the previous paragraph will be useful for our computations.
\end{remark}

It is not difficult to see that $ D $ is exact: When restricted to $ U_l $, it induces an equivalence of categories, and on the image of $ U_o $ we have $ D(\shf E)|_{V_p} = \pi_*(\shf E|_{U_o})^\Gamma $. This is exact since $ \pi|_{U_o} $ is a finite morphism and the order of $ \Gamma $ is invertible.

Let us define the sets of interest (see \cite[2.3]{VaVa}, \cite[chapter 2]{Nakajimabook}):
\begin{definition}\label{def:Xrv}
	Let $r\in \N, \mathbf{v}\in \N^{Q_{\Gamma, 0}} $.
	We write $ X_{r, \mathbf{v}} $ for the set of isomorphism classes of $ \OO_{l_\infty}^r $-framed $ \OO_{\P^2} $-modules $ (\shf E, \phi_{\shf E}) $ where $ \shf E $ is a $ \Gamma $-equivariant sheaf of $ \OO_{\P^2} $-modules, such that $ \shf E $
	
	\begin{itemize}
		\item is coherent,
		\item is torsion-free of rank $ r $,
		\item satisfies $ H^1(\P^2, \shf E\otimes \shf I_{l_\infty})\cong \bigoplus_{\rho_i \in R}(\C^{\mathbf{v}_i}\otimes \rho_i)$, where $ \shf I_{l_\infty} $ is the ideal sheaf of $ l_\infty\subset \P^2 $,
			\end{itemize}
		and where
	 $ \phi_{\shf E}\colon \shf E|_{l_\infty}\isoto \OO_{l_\infty}^{\oplus r} $ is an isomorphism, the \emph{framing isomorphism}.
\end{definition}
	
	It is interesting to note that different sources have slightly different definitions of this set. For instance, Varagnolo and Vasserot (in \cite{VaVa}) do not make the rank an explicit part of the definition. On the other hand, Bruzzo and Sala (in \cite{BruzzoSala}) and Nakajima (in e.g. \cite{Nak02}) add to our definition by requiring local freeness in a neighbourhood of $ l_\infty $. 
In fact, local freeness in a neighbourhood of the framing follows from the other conditions in \cref{def:Xrv}:

\begin{lemma}\label{lem:locFreeAroundFraming}
	A framed sheaf $ \shf E\in X_{r,\mathbf{v}} $ is locally free in a neighbourhood of $ l_\infty $. 
\end{lemma}
\begin{proof}
	Given a reduced noetherian scheme $ X $ equipped with a coherent $ \OO_X $-module $ \shf G $, consider the function \[ \rho_{X, \shf G}(x)= \dim_{k(x)}\shf G_x\otimes_{\OO_{X}}k(x) ,\] where $ k(x)=\OO_{X, x}/\mathfrak{m}_x $. By \cite[Exercise II.5.8]{Har77}, $ \rho_{X, \shf G} $ is an upper semi-continuous function, and $ \shf G $ is locally free on an open subscheme $ U\subset X $ if and only if $ \rho|_U $ is constant.
	
	Let $ \shf E \in X_{r,\mathbf{v}} $.
	It is straightforward to show that for a point $ x\in l_\infty\subset \P^2 $, there is an equality 
	\[ \rho_{\P^2, \shf E}(x)=\rho_{l_\infty, \shf E|_{l_\infty}}(x).\]
	Because $ \shf E $ has rank $ r $, the set $ \{x\in \P^2|\rho_{\P^2,\shf E}(x)\ge r\} $ is dense, and thus all of $ \P^2 $.
	Then, as $ \shf E|_{l_\infty} $ is free, semicontinuity of $ \rho_{\P^2} $ implies that $ \shf E $ is locally free on a neighbourhood of $ l_\infty $.
	
\end{proof}

For comparison with \cref{def:Xrv}, let us restate \cref{def:setYrn} from the Introduction.
\begin{definition}\label{def:Yrn}
	Let $ r, n\in \N $.
	Write $ Y_{r,n} $ for  the set of isomorphism classes of $ \OO_{d_\infty}^r $-framed $ \OO_{\mathcal X} $-modules $ (\shf F, \phi_{\shf F}) $, where $ \shf F $ is a sheaf of $ \OO_{\mathcal X} $-modules such that $ \shf F $
	\begin{itemize}
		\item  is coherent,
		\item is torsion-free of rank $ r $,
		\item satisfies $ \dim H^1(\mathcal X, \shf F\otimes \shf I_{d_\infty})= n $, where $ \shf I_{d_\infty} $ is the ideal sheaf of $ d_\infty\subset \mathcal X $,
			\end{itemize}  
		and where
$ \phi_{\mathcal E} $ is an isomorphism $\phi_{\mathcal E}\colon \shf F|_{d_\infty}\isoto \OO_{d_\infty}^{\oplus r} $.

\end{definition}
When we write $ (\shf E, \phi_{\shf E})\in X_{r,\mathbf v}$ or $ (\shf F, \phi_{\shf F})\in Y_{r,n} $, we mean that we choose a representative for some isomorphism class.
We will often leave the framing isomorphisms $ \phi_{\shf E}, \phi_{\shf F} $ and framing sheaves $ \OO_{d_\infty}^{\oplus r} $ implicit, and simply write $ \shf F \in Y_{r,n} $, $ \shf E\in X_{r, \mathbf v} $, calling both `framed sheaves'.

The conditions on these sets of framed sheaves are clearly parallel. We'll first see what conditions are preserved by $ D(-) $, and then what conditions are preserved by $ \pi^* $. Because $ \pi^*\shf F $ for a framed sheaf $ \shf F\in Y_{r,n} $ may have torsion, we will have to work with the `torsion-free inverse image' functor $ \pi^T $, which we introduce in \cref{def:torFreePullbacl}.

\subsection{Moving from $ X_{r,\mathbf v} $ to $ Y_{r,n} $}

Now fix a sheaf $ \shf E\in X_{r, \mathbf v} $, where $ \mathbf v_0=n $.
Recall the definition of the set $ Y_{r,n} $ from \cref{def:Yrn}.
We shall show that $ D(\shf E) $ is an element of $ Y_{r,n} $, considering each condition separately.
\subsubsection{Coherence}

To see that $ D(\shf E) $ is coherent, we can consider its restrictions to each chart. 
Since $ D(\shf E)|_{V_d} $ is obtained from $ \shf E|_{U_l} $ by descent, it is coherent. 
Furthermore, $ D(\shf E)|_{V_p} = \pi_*(\shf E|_{U_o})^\Gamma$. Since $ \pi $ is a finite map, it preserves coherence (\cite[Exercise II 5.5]{Har77}). Now $\pi_*(\shf E|_{U_o})^\Gamma $ equals the intersection of the kernels of the endomorphisms \[ (1-g)\colon \pi_*(\shf E|_{U_o})\to \pi_*(\shf E|_{U_o}) ,\] where $ g $ runs over all elements of $ \Gamma $. Kernels of morphisms of coherent sheaves on a locally noetherian scheme are coherent, and so are intersections of coherent sheaves. Then $ (\pi_*(\shf E)^\Gamma)|_{U_o} = D(\shf E)_{U_o}  $ is coherent.

Thus $ D(\shf E) $ is coherent.

\subsubsection{Framing}

To see that $ D(\shf E) $ is framed along $ d_\infty $, it is sufficient to restrict the sheaf to ${U_l} $. Then we are dealing with the descent of a framed sheaf to a group quotient. 
Here it is obvious - descent to group quotients commutes with closed immersions, and so $ D(\shf E)|_{d_\infty} $ descends from $\shf E|_{l_\infty}$. Then the framing isomorphism $ \phi_{\shf E}\colon \shf E|_{l_\infty} \isoto \OO_{l_\infty}^{\oplus r} $ induces a framing isomorphism $ D(\shf E)|_{d_\infty} \isoto \OO_{d_\infty}^{\oplus r} $.

\subsubsection{Rank}

Let $ U $ be an open subscheme of $ \P^2 $ on which $ \shf E $ is locally free. Then \[U' = U_f\cap \left(\bigcap_{g\in\Gamma} g(U)\right) \] is a $ \Gamma $-invariant open subscheme such that $ \shf E|_{U'} $ is locally free of rank $ r $.
The morphism $ \pi|_{U'}\colon U'\to U'/\Gamma $ is étale and therefore fpqc, and thus we can apply \cite[Lemma 05B2]{stacks}, stating that the property of being finite locally free is preserved under descent.

So $ D(\shf E)|_{[U'/\Gamma]} $ is locally free of rank $ r $, and so $ D(\shf E) $ has rank $ r. $

\subsubsection{Torsion-freeness}

To show that $ D(\shf E) $ is torsion-free, it is enough to show that $ D(\shf E)|_{U_o} $ is. 

Again, since both $ U_o = \C^2$ and $ V_p= \C^2/\Gamma $ are locally Noetherian schemes, a sheaf on either scheme is torsion-free if and only if it has no subsheaves with support of dimension $ 1 $ or lower. Since $ \pi $ is a finite morphism, $ \pi_* $ preserves the dimension of the support of a sheaf. Thus, if $ \shf E $ is a sheaf on $ U_o $, $ (\pi_*\shf E)|_{V_p} $ is torsion-free if $ \shf E $ is torsion-free. It is clear that $ ((\pi_*\shf E)|_{V_p})^\Gamma $ is torsion-free if $ \pi_*\shf E $ is.
Then $ D(\shf E) $ is torsion-free.

\subsubsection{Cohomology}

By \cref{lem:specseqforGammaStack}, a $ \Gamma $-equivariant sheaf $ \shf E $ on $ \P^2 $ satisfies
\[ H^i([U/\Gamma], D(\shf E)) = (H^i(U, \shf E))^\Gamma \] for any $ \Gamma $-invariant open subscheme $ U\subseteq \P^2 $, since $ |\Gamma| $ is invertible.

\begin{lemma}\label{lem:cohom}
	Let $ \shf E\in X_{r, \mathbf{v}} $ for some dimension vector $ \mathbf{v} $. Let $ n=\mathbf{v}_0 $.
	Then \[ \dim H^1(\mathcal X, D(\shf E)\otimes \shf I_{d_\infty}) = n .\]
\end{lemma}
\begin{proof}
	By \cref{lem:CoarseModSpaceOfX}, both $ r_*$ and $ (r\circ q)_* $ are exact. The same lemma tells us that \[ H^i([\P^2/\Gamma],\shf A) = H^i(\mathcal X, q_*\shf A) = H^i(\P^2/\Gamma, (c\circ q)_*\shf A)\] for any coherent sheaf $ \shf A $ on $ [\P^2/\Gamma] $.
	On $ [\P^2/\Gamma] $, $ \shf E $ descends to a sheaf $ \shf F $. Thus $ \shf E\otimes \shf I_{l_\infty} $ descends (tautologically) to $ \shf F\otimes \shf I_{d_\infty}'$, where $ \shf I_{d_\infty}' $ is the ideal sheaf of $ d_\infty $ as a substack of $ [\P^2/\Gamma] $.  
	It follows by \cref{lem:specseqforGammaStack} that:
	\[H^i(\P^2, \shf F\otimes \shf I_{l_\infty})^\Gamma=H^i([\P^2/\Gamma], \shf F\otimes \shf I_{d_\infty}')=\]
	\[ H^i(\mathcal{X}, q_*(\shf F\otimes \shf I_{d_\infty}')) =H^i(\P^2/\Gamma, (c\circ q)_*(\shf F\otimes \shf I_{d_\infty}')) . \]
	
	It will therefore suffice to show that there is a canonical isomorphism \[ q_*(\shf F\otimes \shf I_{d_\infty}') \isoto D(\shf E)\otimes \shf I_{d_\infty} .\]
	This immediately holds when restricted to $ V_d $.
	Since $ \shf I_{d_\infty}|_{[\C^2/\Gamma]}\cong \OO_{[\C^2/\Gamma]} $, we have \[q_*(\shf F\otimes\shf I_{d_\infty})|_{V_p} =q_*((\shf F\otimes\shf I_{d_\infty})|_{[\C^2/\Gamma]}) \cong q_*(\shf F|_{[\C^2/\Gamma]}), \] and we only have to prove that \begin{equation}\label{eqn:stupid}
		(q_*\shf F)|_{V_p} = (\pi|_{U_o})_*(\shf E|_{U_o})^\Gamma .
	\end{equation} 
	By definition, \[(q_*\shf E)(V_p) = \shf E(V_p\times_{\mathcal X}[\P^2/\Gamma])= \shf E([\C^2/\Gamma]).\] Using \cref{lem:specseqforGammaStack}, this equals $ \shf E(U_o)^\Gamma = \pi_*(\shf E)^\Gamma(V_p) $. Since $ V_p $ and $ U_o $ are affine, this is sufficient to show that \eqref{eqn:stupid} holds.
	
	This concludes the proof.
\end{proof}

We have now shown 

\begin{proposition}\label{prop:DMapsSheaves}
	The functor $ D(-) $ induces a map of sets from the set $ X_{r, n\delta} $ to the set $ Y_{r,n} $.
\end{proposition}

We will see in \cref{prop:mapIsSurjective} that this map is surjective. 

\subsection{Inverse images of elements of $ Y_{r,n} $}
Now let $ \shf F\in Y_{r,n} $. We will not able to show that $ \pi^*\shf F $ lies in $ X_{r, n\delta} $, because: 
\begin{itemize}
	\item $ \pi^*\shf F $ may not be torsion-free, and
	\item it may not hold that $ H^1(\P^2, \pi^*\shf F\otimes \shf I_{l_\infty}) $, considered as a $ \Gamma $-representation, has dimension vector $ n\delta $.
\end{itemize} 
We will instead show that $ \pi^*\shf F $ modulo torsion lies in $ X_{r, \mathbf v} $, where $ \mathbf v $ is some dimension vector satisfying $ \mathbf v_0 = n $. 
Again we consider each criterion for a framed sheaf to belong to $ X_{r, \mathbf{v}} $ separately. 

\subsubsection{Torsion}\label{subsubsec:pullbackTor}
Let us first define the "torsion-free inverse image".

\begin{definition}\label{def:torFreePullbacl}
	We set $ \pi^T(\shf F) = (\pi^*\shf F)/\mathrm{tor}$, where $ \mathrm{tor} $ is the torsion subsheaf of $ \pi^*\shf F $.
\end{definition}
It is clear that $ \pi^T $ is functorial: If $ \sigma\colon \shf F\to \shf G $ is a morphism of sheaves on $ \mathcal X $, then the torsion subsheaf of $ \pi^*\shf F $ is supported on a closed strict subscheme of $ \P^2 $. It follows that the image of this torsion subsheaf under $ \pi^*\sigma $ is also supported on a closed strict subscheme of $ \P^2 $ \ie that it is a torsion subsheaf of $ \shf G $.

\begin{remark}\label{rmk:WhereIsTorsion}
	Since $ \shf F|_{V_d} $ descends from the sheaf $ (\pi^*\shf F)|_{U_l} $, $ \shf F|_{V_d} $ is torsion-free if and only if $ \pi^*\shf F|_{U_l} $ is. It follows that if $ \shf F\in Y_{r,n} $, the torsion subsheaf of $ \pi^*\shf F $ is either zero, or a skyscraper sheaf supported at $ o $.
\end{remark}
\subsubsection{Coherence}
It is enough to work locally, and it follows from {\cite[Prop 5.8]{Har77}} that $ (\pi^*\shf F)|_{U_o} $ is coherent. On the other hand, $ (\pi^*\shf F)|_{U_l} $ is tautologically coherent.

Then $ \pi^*\shf F $ is coherent, and so is $ \pi^T \shf F $.

\subsubsection{Framing}

Suppose that $ \shf F $ has a rank-$ r $-framing $\phi_{\shf F}\colon  \shf F|_{d_\infty}\isoto \OO_{d_\infty}^{\oplus r} $. We can again restrict to $ V_d $. Now $ \shf F|_{V_d} $ is a coherent sheaf on $ V_d $, so by descent $ \shf F $ corresponds to a $ \Gamma $-equivariant sheaf $ \shf E $ on $ U_o $ such that $ \pi^*\shf F|_{V_d} =\shf E$. Then, in the commutative square 

\[ \begin{tikzcd}
	l_\infty \arrow[d, two heads, "\pi|_{l_\infty}"] \arrow[r, hook, "i"] & U_l\arrow[d, two heads, "\pi"] \\
	d_\infty \arrow[r, hook, "j"]                      & {V_d}          
\end{tikzcd}
\]
$ \phi_{\shf F} $ induces the isomorphism \[ i^*\pi^*\shf F = ({\pi|_{l_\infty}})^*j^*\shf F \isoto (\pi|_{l_\infty})^*\OO_{d_\infty}^{\oplus r} = \OO_{l_\infty}^{\oplus r},\]
as required. Finally, because any torsion in $ \pi^*\shf F $ is only supported at $ \{o\} $, we then have \[ i^*\pi^*\shf F=i^*\pi^T\shf F  \isoto \OO_{l_\infty}^{\oplus r}.\]

\subsubsection{Rank}

This is immediate: let $ V \subset \mathcal{X} $ be an open subscheme such that $ \shf F|_V $ is locally free of rank $ r $. Then $ (\pi|_V)^*\shf F|_V $ is locally free of the same rank, and so $ \pi^T\shf F $ has rank $ r $. 

\subsubsection{Cohomology}

Let us show that
\begin{lemma}\label{lem:torCohomology}
	\[ H^1(\P^2, \pi^*\shf F) = H^1(\P^2, \pi^T\shf F), \textrm{ and}\]
	\[ H^1(\P^2, \pi^T\shf F\otimes \shf I_{l_\infty})=H^1(\P^2, \pi^T(\shf F\otimes \shf I_{d_\infty})) \]
\end{lemma}
\begin{proof}
If we again let $ \mathrm{tor} $ be the torsion subsheaf of $ \pi^*\shf F $, the first claim follows from taking the long exact cohomology sequence of \[ 0\to \mathrm{tor}\to \pi^*\shf F\to \pi^T\shf F\to 0 ,\] since $ \mathrm{tor} $ is only supported at $ o $. 
For the second claim, simply note that the inverse image functor commutes with tensor products, and that $ \shf I_{l_\infty} $ is torsion-free, being invertible.
\end{proof}
It follows from the discussion of the preceding paragraphs that if $ \shf F\in Y_{r,n} $, then $ \pi^T\shf F $ is an element of the set $ X_{r, \mathbf{v}} $ for some dimension vector $ \mathbf v $. 

We shall find restrictions on $ \mathbf v $ in \cref{cor:torfreepullbackisgood}.

Finally, we have
\begin{lemma}\label{lem:locFreeAroundFramingX}
	A framed sheaf $ \shf F\in  Y_{r,n} $ is locally free in a neighbourhood of $ d_\infty $.
\end{lemma}
\begin{proof}
	Such a sheaf $ \shf F $ is, on any neighbourhood  $U\supset d_\infty $ that does not contain $ p $, the descent of an equivariant framed sheaf on $ \pi^{-1}(U) $. Now we can simply repeat the proof of \cref{lem:locFreeAroundFraming}.
\end{proof}

\subsection{Correspondences of $ D $ and $ \pi^T $}
Now we investigate the compositions of $ D $ and $ \pi^T $.

First, let $ \shf E\in X_{r, \mathbf{v}} $ for some $ \mathbf{v} $. Consider the sheaf $ \pi^T(D(\shf E)) $. Since $ D(\shf E)|_{U_o} =\pi_*(\shf E|_{V_p})^{\Gamma}$, there is a homomorphism 
\[ \iota\colon \pi^*(D(\shf E))\to \shf E ,\] induced on $ U_o $ by the counit of the adjunction $ \pi^*\vdash \pi_* $, i.e. $ \pi^*\pi_*(\shf E)^\Gamma \to \pi^*\pi_*(\shf E) \to \shf E $. 
Because $ \shf E $ is torsion-free, $ \iota $ factors through $ \pi^T(D(\shf E)) $.

On $ U_l $, $ \iota $ is just the identity. Then, because $ \ker \iota \subset \pi^*(D(\shf E))$  must have support contained in $ \{o\} $, it is torsion.
We have shown:
\begin{lemma}\label{lem:downUpInjection}
	Let $ \shf E\in X_{r,n\delta} $. There is a natural injective morphism
	\[ \pi^T(D(\shf E))\injto \shf E .\]
\end{lemma}

On the other hand, we can show that there is a canonical isomorphism $ D(\pi^*\shf F)\isoto \shf F $. Again, it is enough to show this when restricted to $ V_p $. So let $ N $ be a torsion-free $ \C[x,y]^\Gamma $-module. Then we need to show that there is an isomorphism of $ \C[x,y]^\Gamma $-modules \[ (N\otimes_{{\C[x,y]^\Gamma}}\C[x,y])^\Gamma \isoto  N .\]

To see this, note that an element  $n'\in (N\otimes_{{\C[x,y]^\Gamma}}\C[x,y])^\Gamma $ can be written as a sum of elements of the form \[ \frac{1}{|\Gamma|}\left(n\otimes \sum_{g\in \Gamma}g\cdot a\right) \] for $ n\in N $ and $ a\in \C[x,y] $, which means that $ n' $ lies in the image of the map \[ N\cong N\otimes_{{\C[x,y]^\Gamma}}C[x,y]^\Gamma \to(N\otimes_{{\C[x,y]^\Gamma}}\C[x,y])^\Gamma .\]
So the map $ N\to (N\otimes_{{\C[x,y]^\Gamma}}\C[x,y])^\Gamma $ is surjective. 
On the other hand, since $ \C[x,y]^\Gamma $ is a direct summand of $ \C[x,y] $ considered as a $ \C[x,y]^\Gamma $-module, the map $ N\to N\otimes_{\C[x,y]^\Gamma}\C[x,y] $ is injective.
\footnote{A subring $ R\subset S $ such that for any $ R $-module $ M $, the homomorphism $ M\to M\otimes_R S $ is injective, is called a \emph{pure} subring. As a generalisation of our case, if $ G $ is a linearly reductive affine linear algebraic group acting on a regular Noetherian $ \C $-algebra $ S $, then $ S^G $ is a pure subring of $ S $, see \cite[section 6]{HochRob}.}

This proves the first equality of our next Lemma.

\begin{lemma}\label{lem:upanddowniso}
	Let $ \shf F\in Y_{r,n} $. There are canonical isomorphisms of framed sheaves $ \shf F\isoto D(\pi^*\shf F)\isoto D(\pi^T \shf F) $.
\end{lemma}
\begin{proof}
	We need to show the map $ D(\pi^*\shf F)\to D(\pi^T \shf F) $ induced by $ \pi^*\shf F\to \pi^T\shf F $ is an isomorphism. To do this, let $ \shf T\subset \pi^*\shf F $ be the torsion subsheaf, which has support contained in $ \{o\} $.
	
	Since $ D $ is exact, it maps the exact sequence \[ 0\to\shf T\to\pi^*\shf F\to\pi^T\shf F\to 0 \] to \[ 0\to D(\shf T)\to D(\pi^*\shf F)=\shf F \to D(\pi^T\shf F)\to 0 .\] Now $ \Supp D(\shf T)\subset \{p\} $, which means that $ D(\shf T) $ is a torsion subsheaf of $ \shf F $. But $ \shf F $ is torsion-free. Then $ D(\shf T)=0 $, and we are done.
\end{proof}

Finally, we can prove:
\begin{corollary}\label{cor:torfreepullbackisgood}
	Given $ \shf F\in Y_{r,n} $, we have $ \pi^T\shf F\in X_{r, \mathbf v} $ for some $ \mathbf v $ such that $ \mathbf v_0 =n $.
\end{corollary}
\begin{proof}
	By \cref{lem:upanddowniso}, $ \pi^T\shf F\in X_{r, \mathbf v} $ for some vector $ \mathbf v $. By \cref{lem:cohom}, we have $ \mathbf{v}_0 = n $.
\end{proof}

\section{Proof of the main theorem}\label{sec:VaVa}
In this section, we prove \cref{thm:MainWeaker}.

\subsection{Varagnolo and Vasserot's bijection}
Recall our convention that $ \theta $ denotes an arbitrary stability parameter in the chamber $ C^+ $.
\begin{lemma}[{\cite[Theorem 1]{VaVa}}]\label{thm:VaVa} 
	There is, for any positive integer $ r $ and any nonnegative $ \mathbf v\in Q_0^\N$, a canonical bijection \[ \mathfrak{M}_{\theta}(r, \mathbf{v})(\C)\isoto X_{r, \mathbf{v}} .\]
\end{lemma}
\begin{proof}[Proof sketch of \cref{thm:VaVa}]
	We sketch the map, and state its properties without proof -- see \cite[Theorem 1]{VaVa} and its proof for the full details.
	We will need to use the viewpoint of the quiver variety $ \mathfrak{M}_{\theta}(r, \mathbf{v})$ given in \cref{sec:VaVaViewpoint}\ie we consider the underlying quiver as having only one pair of arrows between the $ 0 $-vertex and the framing vertex $ \infty $, but our quiver representations have an $ r $-dimensional vector space at $ \infty $.
	
	So let $ N = (V_N, W_N, B_N, i_N, j_N)\in \mu^{-1}(0)\subset M(r, \mathbf v)$. We construct the $ \Gamma $-equivariant complex of sheaves
	\begin{equation}
		\label{diag:VaragnoloVasserotConstruction}
		C_N^\bullet ={\left[\begin{tikzcd} \OO_{\P^2}(-l_\infty)\otimes V_N\arrow[r, "a_N"]& \OO_{\P^2}\otimes((V_N\otimes L)\oplus W_N)\arrow[r, "b_N"] &\OO_{\P^2}(l_\infty)\otimes \bigwedge^2 L\otimes V_N \end{tikzcd}\right]}
	\end{equation}
	with 
	\[ a_N = \begin{pmatrix}
		zB-x\cdot e_x-y\cdot e_y\\
		zj
	\end{pmatrix},\quad b_N = \begin{pmatrix}
		zB-x\cdot e_x-y\cdot e_y& zi
	\end{pmatrix}.\] Here $ e_x, e_y $ are the canonical basis vectors of $ L $. We set the terms of $ C_N^\bullet $ to be of degree $ 0,1, $ and $ 2 $.
	
	Then $ C_N^\bullet $ is a complex of $ \Gamma $-equivariant $ \OO_{\P^2} $-modules. 
	We will denote the cohomology objects of the complex $ C_N^\bullet $ by $ \mathcal{H}^i(C_N^\bullet) $.
	Note that the assignment $ N\mapsto C_N^\bullet $ is functorial, i.e. a map $ N\to N' $ of $ \Pi $-modules  (considered as elements of $ M(r, \mathbf{v}) $ from \cref{sec:VaVaViewpoint}) induces a map of complexes $ C_N^\bullet\to C_{N'}^\bullet $. As $ \OO_{\P^2}(i) $ is locally free, this functor is even exact.
	
	For a \emph{stable} $ N $ representing a $ \C $-point of $ \mathfrak{M}_{\theta}(r, \mathbf{v}) $, the complex $ C_N^\bullet $ is even a \emph{monad}, that is, the middle $ \mathcal H^1(C_N^\bullet) $ is the only nonzero cohomology object. This is naturally a $ \Gamma $-equivariant $ \OO_{\P^2} $-module, equipped with a framing, and the map $ N\mapsto \mathcal H^1(C_N^\bullet) $ provides the bijection $ \mathfrak{M}_{\theta}(r, \mathbf{v})(\C)\isoto X_{r, \mathbf{v}} $ (\cite[Theorem 1]{VaVa}).

	We have now made one direction of the bijection from \cref{thm:VaVa} explicit, but we can do so for the inverse direction as well:
	Let $ \shf E\in X_{r, \mathbf{v}} $, set $ V_{\shf E}= H^1(\P^2, \shf E\otimes \shf I_{l_\infty})= H^1(\P^2, \shf E(-l_\infty)) $, and let $ W_{\shf E} $ be an $ r $-dimensional vector space. 
	Futhermore, let $ \shf Q $ be the $ \Gamma $-equivariant locally free sheaf of rank 2 on $ \P^2 $ which is the cokernel in the Euler exact sequence
	\[ 0\to\OO_{\P^2}(-l_\infty) \to \OO_{\P^2}^{\oplus3}\to \shf Q\to 0.\] 
	
	 Associated to any framed $ \Gamma $-equivariant torsion-free sheaf $ \shf E $ on  $ \P^2 $, there is a \emph{Beilinson spectral sequence} (see  \cite[Section 2.1]{Nakajimabook} or \cite{VaVa}) converging to $ \shf E $ determining the quiver data. 
	The $ E_1 $-term of this Beilinson spectral sequence associated to $ \shf E $ is, (after tensoring with $ \OO_{\P^2}(l_\infty) $) the complex 
	\begin{multline*} \OO_{\P^2}(-l_\infty)\otimes H^1(\P^2, \shf E(-l_\infty)){\to} \OO_{\P^2}\otimes H^1(\P^2, \shf E(-l_\infty)\otimes \shf Q\dual)\\ {\to} \OO_{\P^2}(l_\infty)\otimes H^1(\P^2, \shf E(-l_\infty)) 
	\end{multline*} which is isomorphic to a monad 
	\begin{equation*} C_{M_\shf E}^\bullet \coloneqq \left[\begin{tikzcd} \OO_{\P^2}(-l_\infty)\otimes V_{\shf E}\arrow["a_{\shf E}",r]&\OO_{\P^2}\otimes(V_{\shf E}\otimes L\oplus W_{\shf E})\arrow["b_{\shf E}",r]& \OO_{\P^2}(l_\infty)\otimes \Lambda^2 L\otimes V_{\shf E} \end{tikzcd}\right]
	\end{equation*} of the type from \cref{diag:VaragnoloVasserotConstruction}, and so we can read off linear $ \Gamma $-equivariant maps $ i_{M_\shf E}, j_{M_\shf E}, B_{M_\shf E} $.

This makes $ (V_{\shf E}, W_{\shf E}, i_{M_\shf E}, j_{M_\shf E}, B_{M_\shf E}) $ a stable $ \Pi $-module (see \cite[Section 2.1]{Nakajimabook}), and so it determines a (closed) point in $ \mathfrak{M}_\theta(r, \mathbf v) $.
	
\end{proof}

\subsubsection{Notation}
Let us fix some notation: If $ \shf E\in X_{r,\mathbf v} $, we write $ (W_{\shf E}, V_{\shf E},B_{\shf E}, i_{\shf E}, j_{\shf E}) $ for its associated stable $ (Q',\mathcal J) $-representation (up to isomorphism), and $ M_{\shf E} $ for the resulting $ \Pi $-module.
We denote the $ \Gamma $-equivariant $ \OO_{\P^2} $-module that \cref{thm:VaVa} associates to a stable $ \Pi $-module $ M $ by $ \sh M $.

Furthermore, we will write $ \mathcal H^0(C^\bullet_M), \mathcal H^1(C^\bullet_M),\mathcal H^2(C^\bullet_M)  $ for the cohomology objects of the complex $ C^\bullet_M $.

\subsection{Constructing a map $ \mathfrak{M}_{\theta_0}(r, n\delta)(\C)\to Y_{r,n} $}

We now construct a map of sets $ \mathfrak{M}_{\theta_0}(r, n\delta)(\C)\to Y_{r,n} $. 
Let $ x\in \mathfrak{M}_{\theta_0}(r, n\delta)(\C) $. By \cref{lem:VGITsurjectivity}, there is a $ \theta $-stable $ \Pi$-module $ M $ representing a closed point of $ \mathfrak{M}_\theta(r,n\delta)$ that is a lift of $ x $, and this lift corresponds by \cref{thm:VaVa} to a framed, torsion-free rank $ r $- $ \Gamma $-equivariant $ \OO_{\P^2} $- module $ \sh M $.
\begin{proposition}\label{prop:liftingisindep}
	The framed $ \OO_{\mathcal{X}} $-module $ D(\sh M) $ is independent of the choice of $ M $, and is thus only dependent on $ x $.
	In fact, the sheaf $ D(\sh M) $ is only dependent on the $ R_{\theta_0} $-equivalence class of $ x $.
\end{proposition}

We shall need several intermediate lemmas to prove this. The strategy will be to show that the concentrated representation of an $ R_{\theta_0} $-equivalence class gives rise to the same $ \OO_{\mathcal X} $-module as every element of the class. 

\begin{lemma}\label{lem:simplemoduleshavenoinvarianthom}
	
	Suppose that $ M $ is a vertex simple $ \Pi $-module, for some vertex $ i\ne 0 $, corresponding to the simple $ \Gamma $-representation $ \rho_i $. Then we have 
	\[ D(\mathcal H^2(C_M^\bullet))=0.\]
\end{lemma}
\begin{proof}
	We will show that for the étale cover $ \{V_p, V_d\} $ of $ \mathcal{X} $, we have that $ D(\mathcal H^2(C_M^\bullet))|_{V_i} =0 $ for $ i\in \{p,d\} $.
	
	First, for $ V_p $ it is enough to show that $ \pi_*((\mathcal H^2(C_M^\bullet))|_{U_o})^\Gamma =0 $. As $V_p= \C^2/\Gamma $ is an affine scheme, it suffices to show that this sheaf has no nonzero sections defined on all $ V_p $.
	This holds if we can show that 	
	\[ ((\mathcal H^2(C_M^\bullet)(U_o))^\Gamma = 0 .\]
	
	To do this, note that the end of the complex $ C_M^\bullet $ becomes
	\begin{equation*}		
		\begin{tikzcd}
			\OO_{\P^2}\otimes L\otimes \rho_i\arrow[r, "b_M"]& \OO_{\P^2}(l_\infty)\otimes \bigwedge^2L\otimes \rho_i\arrow[r]&0
		\end{tikzcd}  
	\end{equation*}
	where, since $ B_M=0 $, we have $ b_M= (-xe_x-ye_y) $.
	Now $ \bigwedge^2L $ is the trivial representation, and $ \rho_i $ is a nontrivial representation.
	Over any $ \Gamma $-invariant open set $ U\subset \P^2 $ -- for instance $ U_o $ -- a $ \Gamma $-invariant section of $(\OO_{\P^2}(l_\infty)\otimes \bigwedge^2L\otimes \rho_i) $ is of the form 
	\[\begin{split}
		\sum_{j=1}^k((xf_j+yg_j+zh_j)&\otimes (e_x\wedge e_y)\otimes v_j)\\
		=\sum_{j=1}^k(\alpha_j& \otimes (e_x\wedge e_y)\otimes v_j),\\
	\end{split}\] where $ f_j,g_j,h_j $ are sections of $ \OO_{\P^2}(U) $, and $ v_j\in \rho_i $. 
	
	Decompose $ \alpha_j $ by representation type --- i.e. write $ \OO_{\P^2}(U)= \bigoplus_{\rho_i \in R} V_i $, where $ V_i $ is the $ \rho_i $-isotypical component, and set $ \alpha_j = \sum \alpha_{j,i} $ with $ \alpha_{j,i}\in V_i $. Then we see that $ \alpha_{j,0}=0 $, as otherwise $ \sum_{j=1}^k\alpha_j\otimes v_j $ cannot be $ \Gamma $-invariant.
	
	This shows that $ h_j = xh_{j1}+ yh_{j2} $ for some $ h_{j1}, h_{j2} $. It follows that we have \[ (xf_j+yg_j+zh_j)\otimes (e_x\wedge e_y)\otimes v_j = (x(f_j+zh_{j1})+y(g_j+zh_{j2}))\otimes (e_x\wedge e_y)\otimes v_j  .\] But this is the image of \[ \big(-(g_j+zh_{j2})\otimes e_x+(f_j+zh_{j1})\otimes e_y\big)\otimes v_j \] under $ b_M $. Hence $ \coker b_M $ has no $ \Gamma$-invariant sections.

	Second, restrict to $ V_d $: We shall show that $ D(\mathcal H^2(C_M^\bullet))|_{V_d} $ descends from the zero sheaf, and thus is itself zero. So we shall prove that $ \mathcal H^2(C_M^\bullet)|_{U_l} =0$.
	Consider the restriction of $\mathcal  H^2(C_M^\bullet) $ to $ U_l = \P^2/\{o\} $. This scheme is the union of two affine subschemes, namely $ D_+(x) = \P^2\setminus V(x) $ and $ D_+(y)= \P^2\setminus V(y) $.
	We show that $ \mathcal H^2(C_M^\bullet)|_{D_+(x)} =0$, the argument for $ D_+(y) $ is completely analogous. 
	
	On the distinguished open set $ D_+(x) $, any section of $ \OO_{D_+(x)}(l_\infty) $ is a linear combination of elements on the form \[a= f\otimes (e_x\wedge e_y\otimes v) ,\] with $ f\in x\cdot \C[y/x, z/x] $, $ v\in \rho_i $. We can thus write $ f =xg $, with $ g \in \C[y/x, z/x] $.
	Then $ a $ is the image of $ g\otimes e_y\otimes v $ under $ b_M $. It follows that $ b_M|_{D_+(x)} $ is surjective, and so $ H^0({D_+(x)}, \mathcal H^2(C_M^\bullet))=0 $. But as $ \mathcal H^2(C_M^\bullet) $ is a coherent sheaf, and $ D_+(x) $ is affine, this shows that $ \mathcal H^2(C_M^\bullet)|_{D_+(x)} =0 $. 
	
	This concludes the proof.
	
\end{proof}

\begin{lemma}\label{lem:filtration for the kernel}
	Let $ K = (W_K, V_K, B_K, i_K, j_K) $ be a $ (Q', \mathcal J) $-representation (\cref{sec:VaVaViewpoint})such that $ W_K = 0, V_{K,0} =0 $.
	
	Then $ D(\mathcal H^2(C_K^\bullet)) =0$.
\end{lemma}

\begin{proof}
	
	First, note that in this case the map 
	\begin{equation}\label{eq:kernelfiltration}
		b_K  \colon \OO_{\P^2}\otimes L\otimes V\to \OO_{\P^2}\otimes \bigwedge^2L\otimes V
	\end{equation} is given by $ b_k = (B_K-xe_x-ye_y) $.
	There is by \cref{rem:vertex_simples} a filtration 
	\[ 0 = K_0\subset K_1\subset \dots \subset K_n = K, \] where $K_i' = K_i/K_{i-1} $ is a vertex simple (i.e 1-dimensional) $ \Pi $- module.  Thus $ V_{K_i'} = K_i'\otimes \rho_i $ is a simple nontrivial $ \Gamma $-representation, and it follows from \cref{lem:simplemoduleshavenoinvarianthom} that $D(\mathcal H^2(C_{K_i'}^\bullet)) = 0$.
	
	Then the short exact sequences of $ \Pi $-modules
	\[ 0\to K_{i-1}\to K_i\to K_i'\to 0 \] give, through \cref{diag:VaragnoloVasserotConstruction} 
	and applying the exact functor $ D $, sequences 
	\[ D(\mathcal H^2(C_{K_{i-1}}^\bullet))\to D(\mathcal H^2(C_{K_{i}}^\bullet))\to D(\mathcal H^2(C_{K_{i}'}^\bullet)) \] exact in the middle. Induction through these establishes the lemma.
\end{proof}

\begin{proof}[Proof of \cref{prop:liftingisindep}.]
	Let $ M\con $ be the concentrated representation of the $ S_{\theta_0} $-equivalence class of $ x $. Thus $ M\con $ is a stable $ \Pi $-module, and $ \dim M\con\le \dim M = (r, n\delta) $.
	
	Then \cref{lem: concentrated is a quotient} shows that $ M\con $ is a quotient of $ M $, so there is an exact sequence of $ \Pi $-modules 
	\begin{equation}\label{diag:exactseqofPimodules}
		0\to K\to M\to M\con\to 0 \end{equation} 
	for some $ \Pi $-module $ K $, which must satisfy $ \dim_0 K = \dim_\infty K = 0 $.
	
	As $ M $ and $ M\con $ are 
	stable $ \Pi $-modules, we can apply \cref{thm:VaVa} to associate them to $ \Gamma $-equivariant torsion-free $ \OO_{\P^2} $-modules $ \sh M, \sh M_{con} $ respectively.
	
	Let us use the construction \eqref{diag:VaragnoloVasserotConstruction} on the sequence (\ref{diag:exactseqofPimodules}). We get an exact sequence of complexes $ 0\to C_K^\bullet\to C_M^\bullet\to C_{M\con}^\bullet\to 0 $, or if written out:
	\begin{equation}\label{diag:exSeqOfComplexes}
		\begin{tikzcd}
			0\arrow[d]&0\arrow[d]&0\arrow[d]\\
			\OO_{\P^2}(-l_\infty)\otimes V_K\ar[r,"a_K"]\arrow[d]& \OO_{\P^2}\otimes((V_K\otimes L)\oplus W_K)\arrow[r, "b_K"]\arrow[d] &\OO_{\P^2}(l_\infty)\otimes \wedge^2 L\otimes V_K\arrow[d] \\
			\OO_{\P^2}(-l_\infty)\otimes V_M\arrow[r,"a_M"]\arrow[d]& \OO_{\P^2}\otimes((V_M\otimes L)\oplus W_M)\arrow[r, "b_M"]\arrow[d] &\OO_{\P^2}(l_\infty)\otimes \wedge^2 L\otimes V_M \arrow[d]\\
			\OO_{\P^2}(-l_\infty)\otimes V_{M\con}\arrow[r,"a_{M\con}"]\arrow[d]& \OO_{\P^2}\otimes((V_{M\con}\otimes L)\oplus W_{M\con})\arrow[r, "b_{M\con}"] \arrow[d]&\OO_{\P^2}(l_\infty)\otimes \wedge^2 L\otimes V_{M\con} \arrow[d]\\
			0&0&0\rlap{\quad\quad.}
		\end{tikzcd}
	\end{equation}
	Since $ M $ and $ M\con $ are both stable $ \Pi $-modules, \[ \mathcal{H}^0(C_M^\bullet) = \mathcal{H}^0(C_{M\con}^\bullet) =  \mathcal{H}^2(C_M^\bullet) = \mathcal{H}^2(C_{M\con}^\bullet) = 0.\]
	It follows that there is an exact sequence of $ \Gamma $-equivariant $ \OO_{\P^2} $-modules
	\[0\to \mathcal H^1(C_K^\bullet) \to \sh M\to \sh M\con\to \mathcal H^2(C_K^\bullet)\to 0 \]
	where $ \mathcal H^1(C_K^\bullet) =\ker b_K/\im a_K $, $ \mathcal H^2(C_K^\bullet) = (\OO_{\P^2}(l_\infty)\otimes \wedge^2 L\otimes V_K)/\im b_K$.
	
	We now claim that 
	\begin{enumerate}\label{enum:two claims}
		\item $ \mathcal H^1(C_K^\bullet) =0 $, and
		\item $ D(\mathcal H^2(C_K^\bullet)) =0 $.
	\end{enumerate}
	Note that $ \dim_\infty K =0 $, so $ W_K = 0 $, and therefore $ C_K^\bullet $ is the complex
	\begin{equation}\label{eq:good_complex}
		\begin{tikzcd} 
			\OO_{\P^2}(-l_\infty)\otimes V_K\arrow[r, "a_K"]& \OO_{\P^2}\otimes(V_K\otimes L)\arrow[r, "b_K"] &\OO_{\P^2}(l_\infty)\otimes \bigwedge^2 L\otimes V_K 
		\end{tikzcd}
	\end{equation}
	where $ a_K, b_K $ are both given by $ zB-x\cdot e_x-y\cdot e_y $ (interpreted appropriately).
	But this is a Koszul complex \cite[proof of Theorem 1]{VaVa}, so $ H^1(C_K^\bullet) =0 $. This proves claim (1).
	
	Claim (2) follows from \cref{lem:filtration for the kernel}.
	Together, claims (1) and (2) imply that $ D(\sh M) $ and  $ D(\sh M\con )$ are isomorphic.
	Therefore any $ \theta $-stable $ \Pi $-module $ M $ which is $ R_{\theta_0} $-equivalent to $ M_{\con} $ will give rise to the same sheaf on $ \P^2/\Gamma $.
	
\end{proof}

This defines a map of sets $f\colon \mathfrak{M}_{\theta_0}(r,n\delta)(\C)\to Y_{r,n} $: For a point $ x\in \mathfrak{M}_{\theta_0}(r,n\delta)(\C) $, pick any $ M\in \mathfrak{M}_{\theta}(r,n\delta)(\C)$ that maps to $ x $. Then $ f $ maps $ x $ to $ D(\sh M) $.

\subsection{Surjectivity of the map $ X_{r,n\delta}\to Y_{r,n} $}
To begin, we require a small extension of a key proposition from \cite{CGGS}

\begin{lemma}\label{prop:dimensionEstimateStableRep}
	Assume that $ M $ is a $ \theta_{0} $-stable $ \Pi $-module, with $ \dim_\infty M=1 $ (and thus $ \dim_0 M=n $). Then $ \dim M\le (1, n\delta) $.
\end{lemma}
\begin{proof}
	This is simply the proof of \cite[Proposition A.1]{CGGS}, in the case "$ 0\in J$" (notation as in \emph{ibid}). 
\end{proof}

\begin{proposition}\label{prop:mapIsSurjective}
	The map of sets $ X_{r,n\delta}\to Y_{r,n} $, given by $ \shf F\mapsto D(\shf F) $, is surjective.
\end{proposition}

\begin{proof}
	Let $ \shf E\in Y_{r,n} $. By \cref{cor:torfreepullbackisgood}, $ \pi^T\shf E $ is an element of $ X_{r,\mathbf{v}} $ for some dimension vector $ \mathbf{v} $ such that $ {\mathbf{v}_0}=n $.
	
	Now, by \cref{thm:VaVa}, $ \pi^T\shf E $ corresponds to a point of $ \mathfrak M_\theta(r, \mathbf v) $, i.e., a stable $ (Q, \mathcal J) $-representation $ M $ of dimension $ (r,\mathbf{v}) $. Then $ M $ is  $ R_{\theta_0} $-equivalent to a $ \theta_0 $-semistable module $ M\con $. 
	
	By \cref{prop:dimensionEstimateStableRep},
	$\dim M\con$ satisfies \[ \dim_\infty M\con = 1,\quad \dim_1 M\con = n,\quad \dim M\con \le (1, n\delta).\] Then, taking the direct sum of $ M\con $ with the appropriate vertex simple modules will provide a $ \theta_0 $-semistable $ \Pi $-module $ N $ of dimension $ (1, n\delta) $, which lies in the same $ R_{\theta_0} $-equivalence class as $ M\con $.
	Since the map $ \mathfrak{M}_\theta(r, n\delta)\to \mathfrak{M}_{\theta_0}(r, n\delta) $ is surjective by \cref{lem:VGITsurjectivity}, there is a stable $ \Pi $-module $ M'\in \mathfrak{M}_\theta(r, n\delta) $, in the same $ R_{\theta_0} $-equivalence class as $ N $ and $ M\con $. 
	
	It follows from \cref{prop:liftingisindep} that the framed sheaf corresponding to this module is also mapped to $ \shf E $ by $ D $. This concludes the proof.
\end{proof}

\begin{remark}
	We can now show that there is a canonical bijection between $ Y_{1,n} $ and $ \Hilb^n(\C^2/\Gamma)(\C) $.
	Let $ \shf J\in Y_{1,n} $ be a framed sheaf. Then $ \pi^T(\shf J) $ is a framed sheaf lying in $ X_{1,\mathbf v} $ for some $ \mathbf v $ with $ \mathbf v_0=n $.
	By \cite[Proposition 28]{Nakajimabook}, this ideal sheaf can be uniquely identified with a $ \Gamma $-invariant ideal $ I $ of the ring $ \C[x,y] $, such that $ \C[x,y]/I\cong \bigoplus_{\rho_i\in R}\rho_i^{\mathbf v_i}$. 
Then the sheaf $\shf J|_{V_p}= D(\shf I)|_{V_p} $ is an ideal of the ring $ \C[x,y]^\Gamma $ of codimension $ n $, simply given by $ I\cap \C[x,y]^\Gamma $. 

\end{remark}

\subsection{An inverse map}\label{sec:inverseMap}
In this subsection, we construct a map $g\colon Y_{r,n} \to \mathfrak{M}_{\theta_0}(r,n\delta)(\C)$, and show that this is an inverse map to the map $f $ constructed in the proof of \cref{prop:liftingisindep}.

We now return to our original viewpoint (\cref{sec:StabParams}) of the quiver varieties\ie the underlying quiver has $ r $ pairs of arrows between the $ 0 $ and $ \infty $-vertices.

\begin{lemma}\label{lem:inverse exists}
	There is a canonical map $ g\colon Y_{r,n} \to \mathfrak{M}_{\theta_0}(r,n\delta)(\C)$.
\end{lemma}
\begin{proof}
	Let $ \shf F \in Y_{r,n}$. By \cref{lem:cohom} and \cref{cor:torfreepullbackisgood}, there is a vector $ \mathbf{v}\in \N_0^{Q_0} $ satisfying $ \mathbf v_0 =n $ such that $ \pi^T\shf F\in X_{r, \mathbf v} $. This sheaf corresponds by \cref{thm:VaVa} to a stable $ \Pi $-module, which is therefore also $ \theta_0 $-semistable. Therefore the concentrated module $ M\con $ of its $ R_{\theta_0} $-equivalence class satisfies $ {M\con}_{,0} = n $, and by \cref{prop:dimensionEstimateStableRep} also satisfies $ \dim M\con \le (1, n\delta) $. But then $ M\con $ determines a unique $ S_{\theta_0} $-equivalence class of $ \Pi $-modules of dimension $ (1, n\delta) $ -- i.e. a $\C $-valued point of $ \mathfrak{M}_{\theta_0}(r, n\delta) $.
	
	We set this point to be  $ g(\shf F) $.
\end{proof}

It is time to prove:

\begin{theorem}\label{thm:main}
	There is a canonical bijection of sets between $ Y_{r,n} $ and $ \mathfrak{M}_{\theta_0}(r,n\delta)(\C) $.
\end{theorem}
\begin{proof}
	The only part that remains is to prove that the maps $ f $ and $ g $, respectively constructed in \cref{prop:liftingisindep} and \cref{lem:inverse exists} are mutually inverse.

	So let $ x \in \mathfrak{M}_{\theta_0}(r, n\delta)$, and let $ \shf F = f(x) $.
	Then $ g(\shf F)\in \mathfrak{M}_{\theta_0}(r, n\delta) $, but it is not clear whether $ g(\shf F) $ equals $ x$.
	We shall show that any two elements of $ X_{r, n\delta} $ that both are mapped to $ \shf F $ by $ f $, correspond to $ S_{\theta_0} $-equivalent $ \Pi $-modules; this implies that $ g(\shf F)=x $.
	Let $ \{\shf E_i\}_{i\in J} $ be the set of all coherent framed $ \Gamma $-equivariant torsion free sheaves on $ \P^2 $ such that $ D(\shf E_i)= \shf F $. 
	This set is nonempty, as by \cref{prop:mapIsSurjective}, $ \pi^T\shf F $ is a member.
	As $ D(\shf E_i)|_{V_p} = \pi_*(\shf E_i|_{U_o})^\Gamma $, the counit of the adjunction $ \pi_*\vdash \pi^* $ induces for every $ \shf E_i $  a map \[ \pi^*\shf F = \pi^*(D(\shf E_i)) \to \shf E_i. \] Since $ \shf E_i $ is torsion-free, this map annihilates the torsion subsheaf of $ \pi^*\shf F $, and we obtain for every $ i $ a map $ \tau_i\colon \pi^T\shf F\to \shf E_i $.
	
	Set $ \shf E=\pi^T\shf F $. Then for any $ i $, the morphism $\tau_i\colon \shf E \to \shf E_i $ is injective by \cref{lem:downUpInjection}.
	Let $ \shf K_i \coloneqq \coker \tau_i $, a torsion sheaf supported at $ o $. 
	For any $ \shf E_i $, the short exact sequence \[ 0\to \shf E\to \shf E_i\to \shf K\to 0 \]
	induces a sequence of complexes \[  C^\bullet_{M_{\shf E}}\to C^\bullet_{M_{\shf E_i}}\to  C^\bullet_{M_{\shf K}} \] exact in the middle. (Note that the construction of the complex $ C^\bullet_{M_{\shf K}} $ works just as for $ \shf E $, even though $ \shf K $ is torsion.)
	Since $ \shf K $ is supported only at $ o $, we have \[ H^1(\P^2, \shf K_i(-l_\infty)) = H^1(\P^2, \shf K_i(-l_\infty)\otimes \shf Q\dual) =0 ,\]
	and thus $ C^\bullet_{M_{\shf K}} =0$. It follows that there is for every $ i $ a surjection of complexes $ C^\bullet_{M_{\shf E}}\to C^\bullet_{M_{\shf E_i}} $. 
	
	Hence we have surjections of $ \theta_0 $-semistable modules $ M_{\shf E}\to M_{\shf E_i} $, which implies that $ M_{\shf E}$ and $M_{\shf E_i} $ are $ R_{\theta_0} $-equivalent -- and so any two $ M_{\shf E_i} $ must also be $ R_{\theta_0} $-equivalent.
	Therefore any two sheaves in $ X_{r, n\delta} \cong \mathfrak{M}_{\theta}(r, n\delta) $ mapped to $ \shf F $ by $ D(-) $ must correspond to $ S_{\theta_0} $-equivalent $ \Pi $-modules, and so they determine the same point of $ \mathfrak{M}_{\theta_0}(r, n\delta) $. But this shows that $ g(f(M)) = M $.

	On the other hand, choose a sheaf $ \shf F\in Y_{r,n} $ and set $ M=g(\shf F) $. By construction, the concentrated $ \Pi $-module $ R_{\theta_0} $-equivalent to $ M $ is $ R_{\theta_0} $-equivalent to the stable module corresponding to $ \pi^T\shf F $. But \cref{prop:liftingisindep} tells us that, for any two sheaves $ \shf E_1, \shf E_2 $ on $ \P^2 $ corresponding to elements of one $ R_{\theta_0} $-equivalence class, we have $ D(\shf E_1)=D(\shf E_1) $. Since $ D(\pi^T(\shf F))= \shf F $, by \cref{lem:upanddowniso}, it follows that $ f(g(\shf F))=\shf F $.
	
	Thus $ f $ and $ g $ are mutually inverse maps. The proof is complete.
\end{proof}

	\appendix
	\section{Surjectivity of morphisms induced by variation of GIT}\label{sec:surjecProof}\label{sec:HighRankCompleteDegen}
We prove \cref{lem:VGITsurjectivity}.

For this Appendix, let us write \[ \mathfrak{M}_\theta =\mathfrak{M}_\theta(r, n\delta),\quad  \mathfrak{M}_{\theta_0} = \mathfrak{M}_{\theta_0}(r, n\delta),\textrm{ and } \mathfrak{M}_0 = \mathfrak{M}_0(r,n\delta).\]

First, we need the following description of the variety $\mathfrak{M}_0(r, \mathbf{v}) $ due to Nakajima: 

Let $ \lambda = (\lambda_1, \dots, \lambda_k) $ be a partition of an integer $ L $, and set \[ \operatorname{Sym}^L_\lambda(\C^2) =  \left\{\sum_{i=1}^r \lambda_i[x_i]\mid x_i\ne 0 \textrm{ and } x_i\ne x_j \textrm{ for } i\ne j \right\} .\]

\begin{lemma}[{\cite[Equation 4.1]{Nak02}}]\label{lem:NakHigherRankDesc}

	\begin{equation}\label{eq:decomp} \mathfrak{M}_0(r, \mathbf v) = \bigsqcup_{\substack{{\lambda, \mathbf{v}'}\\{m\delta+\mathbf{v}'\leq \mathbf v}}}\mathfrak{M}_0^s(r, \mathbf{v}')\times (\operatorname{Sym}^{(m|\Gamma|)}_\lambda \C^2)^\Gamma .\end{equation}
	where $ \lambda = (\lambda_1, \lambda_2,\dots, \lambda_i) $ is a partition of $ m $, and
	the space $ \mathfrak{M}_0^s(r, \mathbf{v}')$ parametrises framed $ \Gamma $-equivariant \emph{locally free} sheaves  $ (\shf E, \phi_{\shf E}) $,  framed along $ l_\infty $, of rank $ r $, satisfying \[ H^1(\P^2, \shf E\otimes \shf I_{l_\infty})\cong\bigoplus \rho_i^{\mathbf{v}'_i} \] as $ \Gamma $-representations.
	Furthermore, the map $ X_{r, \mathbf v}\cong \mathfrak{M}_\theta(r, \mathbf v)\to \mathfrak{M}_0(r, \mathbf v) $ can be identified with \[ (\shf E, \phi_{\shf E})\mapsto ((\shf E\dual\dual, \phi_{\shf E}), \operatorname{Supp}(\shf E\dual\dual/\shf E) ) .\] 
\end{lemma}
	
Here we use \cref{lem:locFreeAroundFraming} to say that $ \shf E $ is locally free in a neighbourhood $ U $ of $ l_\infty $. Then the morphism $ \shf E\to \shf E\dual\dual $ is an isomorphism over $ U $, so $ \phi_{\shf E} $ is also a framing morphism for $ \shf E\dual\dual $, and the support of $ \shf E\dual\dual/\shf E $ must be contained in the affine scheme $ \P^2\setminus l_\infty =\C^2$.

Let $ t\colon \mathfrak{M}_\theta\to \mathfrak{M}_{\theta_0} $ be the projective morphism induced by variation of GIT. We will prove that $ t $ is a resolution of singularities
\begin{proof}[Proof of \cref{lem:VGITsurjectivity}]\label{prop:degenStabilCond}
	Our strategy is to show that the subscheme $ \mathfrak{M}_{\theta_0}^s $ consisting of $ \theta_0 $-stable modules is nonempty, and therefore dense in $ \mathfrak M_\theta $. By general results on Nakajima quiver varieties (see \eg \cite[2.6]{Nakajima94}), $ \mathfrak{M}_\theta $ is smooth, and if $ \mathfrak{M}_{\theta_0}^s\ne \emptyset $, \[ \dim \mathfrak{M}_{\theta_0}^s = \dim \mathfrak{M}_{\theta} = 2rn.\]
	
	Since a $ \theta_0 $-stable module is $ \theta $-stable, $ \mathfrak{M}_{\theta_0}^s\ne \emptyset $ will imply that $ t $ is birational, and since $ t $ is projective, also that it is surjective. Since  $ \mathfrak{M}_\theta $ is smooth, $ t $ will then be a resolution of singularities.
	
	So it is enough to find an element of $ \mathfrak{M}_{\theta_0}^s $.
	
	It follows from \cref{lem:NakHigherRankDesc} and \cref{rem:vertex_simples}, that there is a decomposition
	\begin{equation} \mathfrak{M}_0 = \bigsqcup_{\substack{{m, m'}\\{m\delta+m'\delta = n\delta}}}\mathfrak{M}_0^s(r, m'\delta)\times \Sym^m(\C^2/\Gamma).\end{equation}

	Let us then choose \[ S\subset (\Sym^{(n|\Gamma|)}_{(1,1,\dots, 1)} \C^2)^\Gamma\subset \Sym^n(\C^2/\Gamma), \] consisting of $ n $ distinct $ \Gamma $-orbits of $ |\Gamma| $ points in $ \C^2 $ (thus no orbit including the origin), and let $ \shf I_S$ be the ideal sheaf of these $ n|\Gamma| $ points, considered as a subscheme of $ \P^2 $. 
	Then the sheaf $\shf E= \shf I_S\oplus \OO_{\P^2}^{\oplus r-1} $ is an element of $ X_{r,n\delta}\cong \mathfrak{M}_\theta $, mapping to $ x =( (\OO_{\P^2}^{\oplus r}, \id_{\OO_{l_\infty}}), S) $ in $ \mathfrak{M}_0 $. Especially $ \shf E $ corresponds, through \cref{thm:VaVa}, to a $ \theta $-stable $ \Pi $-module $ M_\shf E $.
	
	Now note that an element $ (p_1+p_2+\dots+p_n)\in \operatorname{Sym}^n(\C^2/\Gamma) $ corresponds to a direct sum of $ n $ simple $ \Pi_\Gamma $-modules. Each of these simple modules is either a vertex simple supported at $ 0 $ (for a $ p_i $ equal to the singular point $ p\in \C^2/\Gamma $), or has dimension $ \delta $ (for a $ p_i $ different from $ p$).
	The $ \Pi $-module corresponding to $ \OO_{\P^2}^{r-1} $, as an element of $ X_{r, 0} $, is the vertex simple module supported at $ \infty $.
	
	Thus the $ 0 $-polystable (i.e., semisimple) $ \Pi $-module corresponding to $ x \in \mathfrak{M}_0$ can thus be written as a direct sum \[M_{x}\coloneqq N_\infty \oplus N_1\oplus\dots\oplus N_n, \] where $N_\infty $ is a vertex simple module, supported only at $ \infty $, and each $ N_i $ is simple, of dimension vector $ (0, \delta) $. Especially, there is no vertex simple summand supported at another vertex than $ \infty $.
	This means that the module $ t(M_{\shf E}) $ is $ \theta_0 $-stable. For if it was not, its $ \theta_0 $-polystable equivalent would, by \cref{cor:dimConcentratedModule}, be the direct sum of a concentrated module and vertex simple modules supported away from the $ 0 $-vertex. Then the image $ M_x $ of $ M_{\shf E} $ in $ \mathfrak{M}_0 $ would also have vertex simple summands away from this vertex, which contradicts its construction.
	
	Thus $ \mathfrak{M}_{\theta_{0}}^s $ is nonempty. This concludes the proof of (2) of \cref{lem:VGITsurjectivity}
	\end{proof}
\begin{remark}
Using very similar ideas, we can also show that $ C^+_{r,n\delta} $ is a genuine chamber in the wall-and-chamber structure on $ \Theta_{r,n\delta} $, which also follows from \cite{BellamyCrawSchedler}:

	Let $ r>1 $, and let $ I\subset Q_{\Gamma, 0} $. Set $ \theta_I\in \Theta_{r, n\delta} $ to be any stability parameter such that $ \theta_{I}(\rho_\infty)<0, \theta_{I}(\rho_i)>0 $ for $ i\in I $, and $ \theta_{I}(\rho_j)=0 $ for other $ j $. Any parameter in $ \ol{C^+_{r,n\delta}}\setminus C^+_{r,n\delta} $ is on this form for some $ I\ne Q_{\Gamma, 0} $.

We will show that $ \mathfrak{M}_{\theta_I}(r, n\delta) $ is singular\ie $ \theta_I $ is not a generic stability parameter in $ \Theta_{r, \mathbf v} $. 

Let $ \ngammahilb $ be the moduli space of $ \Gamma $-invariant ideals $ I\subset \C[x,y] $ such that $ \C[x,y]/I\isoto \bigoplus_{\rho_i\in R} \rho_i^{n\delta_i}$. Then $ \ngammahilb $ is a smooth quasiprojective scheme, and we have an isomorphism $ \ngammahilb\isoto \mathfrak{M}_{\theta}(1, n\delta) $ \cite{Wang99} (inducing a bijection $ \ngammahilb(\C)\cong X_{1, n\delta} $).
We start by noting that the morphism $\pi_I\colon \ngammahilb\isoto \mathfrak{M}_{\theta}(1, n\delta) \to \mathfrak{M}_{\theta_I}(1, n\delta) $, induced by variation of GIT, is surjective (for instance by a straightforward adaptation of the proof above).
		
		So choose some $ M_I\in \mathfrak{M}_{\theta}(1, n\delta) $, which is not $ \theta_I $-stable, meaning that there is a $ \Pi $-submodule $ N\subset M_I $ such that $ N $ is supported on $ I $. $ M_I $ exists because $ \theta_I $ is not a generic stability parameter in $ \Theta_{1, n\delta} $ (see \eg \cite[Example 2.1]{BelCra})
		Now set $ \shf{I}\in \ngammahilb $ to be the ideal sheaf corresponding to $ M_I $. Then $ \shf I\oplus \OO_{\P^2}^{\oplus r-1} $ is a point of $ X_{r, n\delta} $, corresponding to a $ \theta $-stable module $ M $. By the definition of $ M $, precisely one of the arrows $ b_i\colon \infty\to 0 $ in $ Q_1 $ will be represented by a nonzero map in $ M $. It follows that $ N $ will also be a submodule of $ M $. Thus $ M $ is not $ \theta_{I} $-stable.
		
		Finally, we need to show that a general element of $ \mathfrak{M}_{\theta_{I}}(r, n\delta) $ corresponds to a $ \theta_{I} $-stable $ \Pi $-module.
		So choose $ \shf I'\in \ngammahilb $ such that the $ \theta $-stable $ \Pi $-module corresponding to $ \shf I' $ is also $\theta_I $-stable. Then the $ \Pi $-module corresponding to $ \shf I'\oplus \OO_{\P^2}^{\oplus r-1} $ will also be $ \theta_{I} $-stable.
		
		It follows that $ C^+_{r,n\delta} $ is a chamber in $ \Theta_{r, n\delta} $.
\end{remark}

\renewcommand{\bibname}{References}

\printbibliography
\end{document}